\documentclass[oneside,english]{amsart}
\usepackage[T1]{fontenc}
\usepackage[latin9]{inputenc}
\usepackage{mathrsfs}
\usepackage{amsthm}
\usepackage{amstext}
\usepackage{amssymb}
\usepackage{esint}
\usepackage[all]{xy}

\makeatletter
\numberwithin{equation}{section}
\numberwithin{figure}{section}
\theoremstyle{plain}
\newtheorem{thm}{\protect\theoremname}
  \theoremstyle{plain}
  \newtheorem{conjecture}[thm]{\protect\conjecturename}
  \theoremstyle{plain}
  \newtheorem*{thm*}{\protect\theoremname}
  \theoremstyle{definition}
  \newtheorem{defn}[thm]{\protect\definitionname}
  \theoremstyle{plain}
  \newtheorem{lem}[thm]{\protect\lemmaname}
  \theoremstyle{plain}
  \newtheorem{prop}[thm]{\protect\propositionname}
  \theoremstyle{remark}
  \newtheorem{rem}[thm]{\protect\remarkname}
  \theoremstyle{plain}
  \newtheorem*{conjecture*}{\protect\conjecturename}
  \theoremstyle{plain}
  \newtheorem{cor}[thm]{\protect\corollaryname}

\usepackage{enumerate}
\usepackage{fullpage}
\date{}

\makeatother

\usepackage{babel}
  \providecommand{\conjecturename}{Conjecture}
  \providecommand{\corollaryname}{Corollary}
  \providecommand{\definitionname}{Definition}
  \providecommand{\lemmaname}{Lemma}
  \providecommand{\propositionname}{Proposition}
  \providecommand{\remarkname}{Remark}
  \providecommand{\theoremname}{Theorem}
\providecommand{\theoremname}{Theorem}

\begin{document}

\title{Congruence relations for Shimura varieties associated to $GU(n-1,1)$}

\author{Jean-Stefan Koskivirta}
\begin{abstract}
We prove the congruence relation for the mod-$p$ reduction of Shimura
varieties associated to a unitary similitude group $GU(n-1,1)$ over
$\mathbb{Q}$, when $p$ is inert and $n$ odd. The case when $n$
is even was obtained by T. Wedhorn and O. Bültel in \cite{BW1}, as
a special case of a result of B. Moonen in \cite{BM}, when the ordinary
locus of the $p$-isogeny space is dense. This condition fails in
our case. A key element is the understanding of the supersingular
locus, for which we refer to the two articles \cite{V1,V2}. The proof
makes extensive use of elementary algebraic geometry, but also some
deeper results.
\end{abstract}
\maketitle

\section*{Introduction}

Let $(G,X)$ be a Shimura datum where $G$ is a reductive group over
$\mathbb{Q}$. We fix a prime $p$. For every compact open subgroup
$K\subset G(\mathbb{A}_{f})$, let $Sh_{K}$ the associated Shimura
variety with reflex field $E$. The complex points of $Sh_{K}$ are:
\[
Sh_{K}(\mathbb{C})=G(\mathbb{Q})\diagdown\left(X\times G(\mathbb{A}_{f})\diagup K\right)
\]
When $K$ is sufficiently small, $Sh_{K}$ is smooth. Assume $G_{\mathbb{Q}_{p}}$
is unramified and $K=K_{p}K^{p}$ with $K_{p}\subset G(\mathbb{Q}_{p})$
hyperspecial and $K^{p}\subset G(\mathbb{A}_{f}^{p})$, then $Sh_{K}$
is said to have good reduction at $p$. Let $\mathfrak{p}$ be a prime
in $E$ lying over $p$. In \cite{BR}, the authors define a polynomial
$H_{\mathfrak{p}}$ with coefficients in the Hecke algebra $\mathcal{H}(G(\mathbb{Q}_{p})//K_{p})$,
the set of $\mathbb{Q}$-linear combinations of $K_{p}$-double cosets
of $G(\mathbb{Q}_{p})$. It is made into a ring by convolution. This
ring acts on the cohomology of the Shimura variety. Denote by $Fr_{\mathfrak{p}}$
the conjugacy class of geometric Frobenius in $\textrm{Gal}(\overline{\mathbb{Q}}/E)$.
Blasius and Rogawski conjectured the following:
\begin{conjecture}
\label{rdcconj}Let $\ell$ be a prime $\neq p$. Inside the ring
\textup{$\textrm{End}_{\mathbb{Q}_{l}}\left(H_{et}^{\bullet}(Sh_{K}\times_{E}\overline{\mathbb{Q}},\mathbb{Q}_{\ell})\right)$},
the following relation holds

\textup{
\[
H_{\mathfrak{p}}(\textrm{Fr}{}_{\mathfrak{p}})=0
\]
}
\end{conjecture}
This equation makes sense since the action of Galois commutes with
that of $\mathcal{H}(G(\mathbb{Q}_{p})//K_{p})$. In the PEL case,
an integral model over $\mathcal{O}_{E_{\mathfrak{p}}}$ can be defined
explicitely, and the cohomology of $\overline{Sh_{K}}=Sh_{K}\times\kappa(\mathcal{O}_{E_{\mathfrak{p}}})$
coincides in many cases with that of $Sh_{K}\times E$.

\medskip{}

In the case of Shimura curves, this conjecture was proved by Eichler,
Shimura and Ihara, and was used to determine completely the eigenvalues
of $\textrm{Fr}_{\mathfrak{p}}$ acting on $H_{et}^{i}(Sh_{K}\times_{E}\overline{\mathbb{Q}},\mathbb{Q}_{l})$.
More general situations have been dealt with: T. Wedhorn proved conjecture
\ref{rdcconj} in the PEL case for groups that are split over $\mathbb{Q}_{p}$
in \cite{W1}, O. Bültel for certain orthogonal groups in\cite{bult},
and they worked out together the unitary case of signature $(n-1,1)$
with $n$ even in \cite{BW1}.

\medskip{}

In the latter article, the authors use a moduli space $p\textrm{-}\mathscr{I}sog$
which parametrizes $p$-isogenies between points of $Sh_{K}$. It
was first introduced in \cite{fc}. It comes with two maps $s,t$
to $Sh_{K}$, associating an isogeny to its source and target respectively.
For any field $L$ with a map $\mathcal{O}_{E_{\mathfrak{p}}}\rightarrow L$,
we consider the $\mathbb{Q}$-algebra of cycles in $p\textrm{-}\mathscr{I}sog\times L$,
where multiplication is defined by composition of isogenies, and we
denote by $\mathbb{Q}\left[p\textrm{-}\mathscr{I}sog^{ord}\times L\right]$
the subalgebra generated by the irreducible components.

\medskip{}

In \cite{Wed1999} and \cite{BM}, the authors define the $\mu$-ordinary
locus in the good reduction of a PEL Shimura variety. It is at the
same time a Newton polygon stratum and an Ekedahl-Oort stratum. Furthermore,
it posesses a unique isomorphism class of $p$-divisible groups. It
can also be defined as the unique open stratum in each of these stratifications.
We will denote by $\overline{Sh_{K}}^{ord}$ the ordinary locus. Define
the ordinary locus $p\textrm{-}\mathscr{I}sog^{ord}\times\kappa(\mathcal{O}_{E_{\mathfrak{p}}})$
of $p\textrm{-}\mathscr{I}sog\times\kappa(\mathcal{O}_{E_{\mathfrak{p}}})$
by taking inverse image by $s$ (or $t$). Finally define $\mathbb{Q}\left[p\textrm{-}\mathscr{I}sog^{ord}\times\kappa(\mathcal{O}_{E_{p}})\right]$
in the same fashion as hereabove. We have a diagram of $\mathbb{Q}$-algebra
homomorphisms:

\[
\xymatrix{\mathcal{H}_{0}(G(\mathbb{Q}_{p})//K_{p})\ar[dd]_{\dot{S}}\ar[r]^{h^{0}} & \mathbb{Q}\left[p\textrm{-}\mathscr{I}sog\times E\right]\ar[d]^{\sigma}\\
 & \mathbb{Q}\left[p\textrm{-}\mathscr{I}sog\times\kappa(\mathcal{O}_{E_{\mathfrak{p}}})\right]\ar@<1ex>[d]^{\textrm{ord}}\\
\mathcal{H}_{0}(M(\mathbb{Q}_{p})//(K_{p}\cap M(\mathbb{Q}_{p})))\ar[r]_{\overline{h}} & \mathbb{Q}\left[p\textrm{-}\mathscr{I}sog^{ord}\times\kappa(\mathcal{O}_{E_{\mathfrak{p}}})\right]\ar@<1ex>[u]^{\textrm{cl}}
}
\]
The big square is commutative. Here $M\subset G_{\mathbb{Q}_{p}}$
is the centralizer of the norm of the minuscule coweight $\mu$ of
$G$ associated to $(G,X)$. The algebras on the left hand side of
the diagram are subalgebras of the Hecke algebras containing functions
with integral support. The morphism $\dot{S}$ is a twisted version
of the Satake homomorphism. The map $\sigma$ is specialization of
cycles, the map $\textrm{ord}$ intersects a cycle in $p\textrm{-}\mathscr{I}sog\times\kappa(\mathcal{O}_{E_{\mathfrak{p}}})$
with the ordinaru locus, and the map $\textrm{cl}$ is simply defined
by taking the closure of a cycle of $p\textrm{-}\mathscr{I}sog^{ord}\times\kappa(\mathcal{O}_{E_{\mathfrak{p}}})$.
There is a natural Frobenius section of $s$ defined on $\overline{Sh_{K}}$,
defined by mapping an abelian variety to its Frobenius isogeny, which
produces a closed subscheme $F$ of $p\textrm{-}\mathscr{I}sog\times\kappa(\mathcal{O}_{E_{p}})$.
This subscheme is ordinary, in the sense that $\textrm{cl}\circ\textrm{ord}(F)=F$.
In this context, by ``congruence relation'' we mean the conjecture:
\begin{conjecture}
\label{rdcconjpisog}Consider the polynomial $H_{\mathfrak{p}}$ inside
$\mathbb{Q}\left[p\textrm{-}\mathscr{I}sog\times\kappa(\mathcal{O}_{E_{p}})\right]$
via the morphism $\sigma\circ h^{0}$. The element $F$ lies in the
center of this ring and the following relation holds:\textup{
\[
H_{\mathfrak{p}}(F)=0
\]
}
\end{conjecture}
It is related to conjecture \ref{rdcconj} by using functorial properties
of cohomology. The geometric relation $H_{\mathfrak{p}}(F)=0$ implies
the same equality on the cohomology. B. Moonen has shown the ``ordinary''
congruence relation:
\begin{thm}
\label{moonenRDCord}Consider the polynomial $H_{\mathfrak{p}}$ inside
$\mathbb{Q}\left[p\textrm{-}\mathscr{I}sog^{ord}\times\kappa(\mathcal{O}_{E_{p}})\right]$
via the morphism \textup{$\textrm{ord}\circ\sigma\circ h^{0}$}. In
this ring, the following relation holds:\textup{
\[
H_{\mathfrak{p}}(F)=0
\]
}
\end{thm}
When the ordinary locus of $p\textrm{-}\mathscr{I}sog\times\kappa(\mathcal{O}_{E_{p}})$
is dense, this theorem is equivalent to the congruence relation. This
condition is satisfied in almost all the examples where conjecture
\ref{rdcconjpisog} is known. In the unitary similitude case, this
density condition is satisfied if and only if $n$ is even. As Moonen
points out, the condition fails for Hilbert-Blumenthal varieties.
However, conjecture \ref{rdcconjpisog} can be proved in this case
by the following argument. The Hecke polynomial factors into a product
$H_{\mathfrak{p}}=PQ$ such that $P$ annihilates $F$ on the ordinary
locus, and $P$ is ordinary, that is $\textrm{cl}\circ\textrm{ord}(P)=P$.
The result follows immediately.

\medskip{}

In this article, we prove conjecture \ref{rdcconjpisog} in the unitary
similitude case $GU(n-1,1)$ when $n$ is odd. We first show that
the Hecke polynomial factors into a product $H_{p}(t)=R(t)\cdot(t-p^{n-1}1_{pK_{p}})$.
Let $H'_{p}$ and $R'$ be the polynomials obtained by applying the
map $\textrm{cl}\circ\textrm{ord}$ to $H_{p}$ and $R$. We get:
$\left(H_{p}-H'_{p}\right)(t)=\left(R-R'\right)(t)\cdot\left(t-p^{n-1}\left\langle p\right\rangle \right)$,
where $\left\langle p\right\rangle $ is the multiplication-by-$p$
cycle. It is the image of $1_{pK_{p}}$ by $\sigma\circ h^{0}$. Because
of theorem \ref{moonenRDCord}, we have $H'_{p}(F)=0$ inside $\mathbb{Q}\left[p\textrm{-}\mathscr{I}sog\times\kappa(\mathcal{O}_{E_{p}})\right]$.
Therefore, conjecture \ref{rdcconjpisog} boils down to $\left(H_{p}-H'_{p}\right)(F)=0$.
The factor $\left(R-R'\right)(F)$ is a cycle formed only of supersingular
components of $p\textrm{-}\mathscr{I}sog\times\kappa(\mathcal{O}_{E_{p}})$.
The final argument is the following result:
\begin{thm*}
Let $C\subset p\textrm{-}\mathscr{I}sog\times\kappa(\mathcal{O}_{E_{p}})$
be a supersingular irreducible component. Then:
\[
C\cdot(F-p^{n-1}\left\langle p\right\rangle )=0
\]
inside \textup{$\mathbb{Q}[p\textrm{-}\mathscr{I}sog\times\kappa(\mathcal{O}_{E_{p}})]$.}
\end{thm*}
We will now give an overview on how this article is organized. In
the first section, we establish the factorization of the Hecke polynomial.
In the second one, we give the main facts concerning the Shimura variety.
All of them can be found in \cite{BW1}, we choose to repeat them
here for the sake of clarity. Section 3 is dedicated to the moduli
space of $p$-isogenies. Section 4 studies the supersingular locus
of $p\textrm{-}\mathscr{I}sog\times\kappa(\mathcal{O}_{E_{p}})$.
Here we use mainly \cite{V1,V2} and also \cite{BW1} for some key
results. Finally, in section 5, we prove conjecture \ref{rdcconjpisog}.

\section*{Notations}
\begin{enumerate}
\item We fix $n\geq3$ an \textbf{odd integer}, $k=\frac{n+1}{2}$ and $p>2$
a prime. Let $\mathbb{\overline{Q}}_{p}$ be an algebraic closure
of $\mathbb{Q}_{p}$. We denote by $\mathbb{\overline{Q}}$ the algebraic
closure of $\mathbb{Q}$ inside $\mathbb{C}$. We fix an embedding
$\mathbb{\overline{Q}}\hookrightarrow\mathbb{\overline{Q}}_{p}$. 
\item Let $E$ be an imaginary quadratic extension of $\mathbb{Q}$, such
that $p$ is \textbf{inert} in $E$. We write $\sigma:x\mapsto\overline{x}$
for the non trivial automorphism of $E$ and $E_{p}$ for the completion
of $E$ at $p$. Let $\mathcal{O}_{E_{p}}$ be the ring of integers
of $E_{p}$ and $\kappa(\mathcal{O}_{E_{p}})=\frac{\mathcal{O}_{E_{p}}}{p\mathcal{O}_{E_{p}}}$
the residual field.
\item We fix an embedding $\vartheta:E\hookrightarrow\mathbb{\mathbb{\overline{Q}}}$.
We denote by $\overline{\mathbb{F}}$ the algebraic closure of $\kappa(\mathcal{O}_{E_{p}})$
provided by the embedding $E\hookrightarrow\mathbb{\overline{Q}}_{p}$.
We choose an element $\alpha\in E^{\times}\cap\mathcal{O}_{E_{p}}^{\times}$
such that the imaginary part of $\alpha$ is $>0$ and $\alpha+\overline{\alpha}=0$.
If $z\in\mathbb{C}$ and $z=a+\alpha b$, $a,b\in\mathbb{R}$, we
call $b$ the $\alpha$-imaginary part of $z$.
\item $(V,\psi)$ is a hermitian space of dimension $n$, i.e $V$ is an
$n$-dimensional $E$-vector space, and $\psi:V\times V\rightarrow E$
a non degenerate hermitian pairing. \textbf{We assume the signature
of $(V,\psi)$ to be $(n-1,1)$.}
\item $G$ is the (connected, reductive) algebraic $\mathbb{Q}$-group of
unitary similitudes of $(V,\psi)$.
\item Let $\mathcal{B}_{W}=(e_{1},...,e_{n})$ be a Witt basis of $V\otimes\mathbb{Q}_{p}$.
This means $V\otimes\mathbb{Q}_{p}=V_{0}\oplus\bigoplus_{1\leq i\leq k}H_{i}$
is an orthogonal Witt decomposition, where $V_{0}=\textrm{Vect}_{E_{p}}(e_{k})$
is anisotropic, $H_{i}=\textrm{Vect}_{E_{p}}(e_{i},e_{n+1-i})$ is
a hyperbolic plane with $\psi(e_{i},e_{n+1-i})=1$.
\item In the basis $\mathcal{B}_{W}$, the diagonal matrices of $G_{\mathbb{Q}_{p}}$
form a torus $T$ and the upper-triangular matrices of $G$ a Borel
subgroup $B$ containing $T$. In the basis $\mathcal{B}$, an element
of $T(\mathbb{Q}_{p})$ has matrix $\textrm{diag}(x_{1},...,x_{n})\in GL_{n}(E_{p})$
with
\[
\overline{x_{1}}x_{n}=\overline{x_{2}}x_{n-1}...=\overline{x_{k}}x_{k}.
\]

\item Let $\Omega(T)$ be the Weyl group of $T$ over $\mathbb{Q}_{p}$.
It is the group of permutations of $\left\{ 1,...,n\right\} $ fixing
the equations above. Thus,
\[
\Omega(T)=\left\{ \sigma\in\mathfrak{S}_{n},\forall i\in\left\{ 1,...,n\right\} ;\sigma(i)+\sigma(n+1-i)=n\right\} .
\]

\item Let $\rho$ be the half-sum of positive roots with respect to $(B,T)$.
\item Let $\Lambda$ be the $\mathcal{O}_{E_{p}}$-lattice generated by
the $e_{i}$. \textbf{We assume that $\psi$ defines a perfect pairing}
$\Lambda\times\Lambda\rightarrow\mathcal{O}_{E_{p}}$. This amounts
to $\psi(e_{k},e_{k})\in\mathbb{Z}_{p}^{\times}$ and implies that
$\textrm{det}(\psi)=1\in\frac{\mathbb{Q}_{p}^{\times}}{N(E_{p}^{\times})}$.
\item Let $K_{p}=\textrm{Stab}_{G(\mathbb{Q}_{p})}(\Lambda)$, it is a hyperspecial
subgroup of $G(\mathbb{Q}_{p})$. We write $L=K_{p}\cap B(\mathbb{Q}_{p})$
and $T_{c}=K_{p}\cap T(\mathbb{Q}_{p})$.
\item Let $\varphi:V\times V\rightarrow\mathbb{Q}$ be the $\alpha$-imaginary
part of $\psi$. Then $\varphi$ is a skew-symmetric form such that
$\forall e\in E$, $\forall x,y\in V$, $\varphi(ex,y)=\varphi(x,\overline{e}y)$.
\end{enumerate}

\section{Hecke polynomial}

\subsection{Unitary similitude group}

There is an isomorphism $V\otimes_{\mathbb{Q}}E\simeq\bigoplus_{\tau\in\textrm{Gal}(E/\mathbb{Q})}V$.
The choice of $\textrm{id}\in\textrm{Gal}(E/\mathbb{Q})$ gives an
isomorphism:
\begin{equation}
G_{E}\simeq GL_{E}(V)\times\mathbb{G}_{m}.\label{dep}
\end{equation}
Let $\mathcal{B}$ be an $E$-basis of $V$ and let $J$ be the matrix
of $\psi$ in $\mathcal{B}$. The group $\textrm{Gal}(E/\mathbb{Q})=\{1,\sigma\}$
acts on $G(E)\simeq GL_{n}(E)\times E^{\times}$ by:

\begin{equation}
\sigma.(A,\lambda)=(\overline{\lambda}J(^{t}\bar{A}^{-1})J,\overline{\lambda}),\quad\forall(A,\lambda)\in GL_{n}(E)\times E^{\times}.\label{action}
\end{equation}

\subsection{Dual group}

For the diagonal torus $T_{n,\mathbb{Q}}\subset Gl_{n,\mathbb{Q}}$,
we denote by $\chi_{1},...,\chi_{n}$ (resp. $\mu_{1},...,\mu_{n}$)
the usual characters (resp. cocharacters) of $T_{n,\mathbb{Q}}$.
Let $\chi_{0}$ (resp. $\mu_{0}$) the character (resp. cocharacter)
of $T_{n,\mathbb{Q}}\times\mathbb{G}_{m,\mathbb{Q}}$ defined by $(A,x_{0})\mapsto x_{0}$
(resp. $x\mapsto(I_{n},x)$).

The dual group of $G$ is $\widehat{G}=GL_{n,\mathbb{C}}\times\mathbb{G}_{m,\mathbb{C}}$.
A splitting is a triplet $\Sigma=(\widehat{T},\widehat{B},\{X_{\alpha}\})$
where $(\widehat{B},\widehat{T})$ is a Borel pair of $\widehat{G}$
and $X_{\alpha}\in\textrm{Lie}(\widehat{G}\text{)}$ an eigenvector
for every simple root $\alpha$ of $\widehat{G}$. The $\textrm{Gal}(E/\mathbb{Q})$-action
on $\Psi(\widehat{G})=\Psi(G)^{\vee}$ lifts uniquely to an automorphism
of $\widehat{G}$ fixing $\Sigma$ (cf. \cite{BR}, section 1.6).
We make the following standard choices: 
\begin{align*}
 & \widehat{T}=\{\textrm{diagonal matrices}\}\times\mathbb{G}_{m,\mathbb{C}}\\
 & \widehat{B}=\{\textrm{upper-triangular matrices}\}\times\mathbb{G}_{m,\mathbb{C}}\\
 & \{X_{k}\}=(\delta_{i,k}\delta_{j,k+1})\textrm{ for }k=1,2,...,n-1.
\end{align*}
The vector $X_{k}$ lies in $\textrm{Lie}(\widehat{G})=M_{n}(\mathbb{C})\oplus\mathbb{C}$
and is an eigenvector for the simple root $\chi_{k}-\chi_{k+1}$ of
$\widehat{T}$. There is a unique non trivial automorphism of $\widehat{G}$
fixing $\Sigma$, giving the action of $\sigma$ on $\widehat{G}$:

\begin{equation}
\begin{array}{ccc}
\widehat{G} & \longrightarrow & \widehat{G}\\
(A,\lambda) & \longmapsto & (J'(^{t}A^{-1})J',det(A)\lambda)
\end{array}
\end{equation}
where $J'=((-1)^{i-1}\delta_{i,n+1-j}){}_{i,j}$ (cf. \textsl{ibid},
1.8(c)).

The choice of the basis $\mathcal{B}_{W}$ gives an identification
between $\left(G_{\overline{\mathbb{Q}_{p}}},T_{\overline{\mathbb{Q}_{p}}}\right)$
and $\left(GL_{n,\overline{\mathbb{Q}_{p}}}\times\mathbb{G}_{m,\overline{\mathbb{Q}_{p}}},T_{n,\overline{\mathbb{Q}_{p}}}\times\mathbb{G}_{m,\overline{\mathbb{Q}_{p}}}\right)$
through (\ref{dep}). We fix the identification $\Psi(\widehat{G},\widehat{T})\simeq\Psi(G,T)^{\vee}$
given by $\chi_{i}\leftrightarrow\mu_{i}$ for $i=0,...,n$. We also
identify $\widehat{T}$ and $\textrm{Hom}(X_{*}(T),\mathbb{C}^{\times})$
such that $\left(\textrm{diag}(x_{1},...,x_{n}),x_{0}\right)\in\widehat{T}$
corresponds to the map $\mu_{i}\mapsto x_{i}$.

\subsection{Shimura datum}

Choose an isomorphism $G_{\mathbb{R}}\simeq G_{J_{0},\mathbb{R}}$.
Consider the morphism $h:\mathbb{S}\rightarrow G_{\mathbb{R}}$ of
algebraic groups over $\mathbb{R}$ defined on $\mathbb{R}$-points
by:

\begin{equation}
\begin{array}{ccc}
\mathbb{S}(\mathbb{R})=\mathbb{C}^{\times} & \longrightarrow & G(\mathbb{R})\simeq G_{J_{0}}(\mathbb{R})\\
z & \longmapsto & \textrm{diag}(z,...,z,\overline{z})
\end{array}.\label{hdds}
\end{equation}
Let $X$ be the $G(\mathbb{R})$-conjugacy class of $h$. Then $(G,X)$
is a Shimura datum (\cite{milneintro} definition 5.5). Its reflex
field is $E$. Composing $h_{\mathbb{C}}$ on the right-hand side
by $\mathbb{G}_{m,\mathbb{C}}\hookrightarrow\prod_{\sigma\in\textrm{Gal}(\mathbb{C}/\mathbb{R})}\mathbb{G}_{m,\mathbb{C}}\simeq\mathbb{S}_{\mathbb{C}}$
(given by $\sigma=\textrm{Id}$) gives a cocharacter $\mu:\mathbb{G}_{m,\mathbb{C}}\rightarrow G_{\mathbb{C}}$.
Finally, write $\widehat{\mu}$ for the associated character of $\widehat{T}$
which is dominant relative to $\widehat{B}$. We have $\widehat{\mu}=\chi_{1}+...+\chi_{n-1}+\chi_{0}$
.

\subsection{The representation $r$}

Let $r$ be the irreducible representation of $GL_{n,\mathbb{C}}\times\mathbb{G}_{m,\mathbb{C}}$
of highest weight $\widehat{\mu}$ relatively to $(\widehat{B},\widehat{T})$.
Let $\rho$ denote the identity representation $GL_{n,\mathbb{C}}\rightarrow GL_{n,\mathbb{C}}$,
its weights are $(\chi_{i}){}_{1\leq1\leq n}$. The representation
$\textrm{det}\otimes\rho^{\vee}$ of $GL_{n,\mathbb{C}}$ is irreducible
and its highest weight is $\chi_{1}+...+\chi_{n-1}=\textrm{det}-\chi_{n}$.
Thus, we can define $r$ as follows: 
\[
\begin{array}{ccccc}
r & : & GL_{n,\mathbb{C}}\times\mathbb{G}_{m} & \longrightarrow & GL_{n,\mathbb{C}}\\
 &  & (A,\lambda) & \mapsto & \lambda\textrm{det}(A)\ ^{t}A^{-1}
\end{array}.
\]

\begin{defn}
\label{poldeHeckedef}The Hecke polynomial associated to $(G,X)$
is:
\begin{equation}
H_{p}(t)=\textrm{det}\left(t-p^{n-1}r(g(\sigma\cdot g)\right).\label{H}
\end{equation}
The coefficients of $H_{p}$ are functions on $\widehat{G}$ invariant
under twisted conjugation $c_{x}:g\mapsto xg(\sigma\cdot x^{-1})$,
for $x\in\widehat{G}$.
\end{defn}

\subsection{Hecke algebra}
\begin{defn}
For any $\mathbb{Q}$-algebra $R$, the Hecke algebra $\mathcal{H}_{R}(G(\mathbb{Q}_{p})//K_{p})$
is the set of $K_{p}$-biinvariant, compactly supported functions
$G(\mathbb{Q}_{p})\rightarrow R$. Multiplication is defined by convolution:

\[
(f\star g)(y)=\int_{G(\mathbb{Q}_{p})}f(x)g(x^{-1}y)dx
\]
where the Haar measure on $G(\mathbb{Q}_{p})$ is normalized by $|K_{p}|=1$.
\end{defn}
We recall some facts about the Satake isomorphism. We identify $\mathbb{Q}\left[X_{*}(A)\right]$
and $\mathcal{H}_{\mathbb{Q}}(T(\mathbb{Q}_{p})//T_{c})$ by $\lambda\mapsto1_{\lambda(p)T_{c}}$
for $\lambda\in X_{*}(A)$. In \cite{W1} (1.7,1.8), the twisted Satake
homomorphism $\dot{S}_{T}^{G}$ is defined by the composition

\[
\mathcal{H}_{\mathbb{Q}}(G(\mathbb{Q}_{p})//K_{p})\rightarrow\mathcal{H}_{\mathbb{Q}}(B(\mathbb{Q}_{p})//L)\rightarrow\mathcal{H}_{\mathbb{Q}}(T(\mathbb{Q}_{p})//T_{c})
\]
where the first arrow is restriction of functions and the second is
the quotient by the unipotent radical of $B$. It induces an isomorphism
between $\mathcal{H}_{\mathbb{Q}}(G(\mathbb{Q}_{p})//K_{p})$ and
the subalgebra of $\mathcal{H}_{\mathbb{Q}}(T(\mathbb{Q}_{p})//T_{c})^{\Omega(T),\bullet}$
(the Weyl group acts by the ``dot action'', see \textit{ibid}, 1.8).
Denote by $S_{T}^{G}:\mathcal{H}_{\mathbb{C}}(G(\mathbb{Q}_{p})//K_{p})\rightarrow\mathcal{H}_{\mathbb{C}}(T(\mathbb{Q}_{p})//T_{c})$
the usual Satake isomorphism. Then $S_{T}^{G}=\alpha\circ\dot{S}_{T}^{G}$
where $\alpha:\mathbb{C}\left[X_{*}(A)\right]\rightarrow\mathbb{C}\left[X_{*}(A)\right]$
is defined by $\nu\mapsto p^{-2\left\langle \rho,\nu\right\rangle }\nu$.

\subsection{Hecke polynomial}

The coefficients of $H_{p}$ are polynomial functions on $\widehat{G}$
invariant under twisted conjugation. This is the same as polynomial
functions on $\widehat{T}$ invariant under $\Omega(T)$ and twisted
conjugation, or polynomial functions on $\widehat{A}$ invariant under
$\Omega(T)$. By the untwisted Satake isomorphism, they correspond
to elements in $\mathcal{H}_{\mathbb{Q}}(G(\mathbb{Q}_{p})//K_{p})$.
\begin{lem}
\label{pdansH}The function $\widehat{G}\rightarrow\mathbb{C}$ given
by \textup{$(A,x)\mapsto\textrm{det}(A)x^{2}$} is invariant under
twisted conjugation. It corresponds to the element $1_{pK_{p}}$ in
$\mathcal{H}_{\mathbb{C}}(G(\mathbb{Q}_{p})//K_{p})$.\end{lem}
\begin{proof}
The element $1_{pK_{p}}$ maps to $1_{pT_{c}}$ by the Satake isomorphism.
This element corresponds to $\lambda\in\mathbb{Q}\left[X_{*}(A)\right]$
where $\lambda$ is the cocharacter $u\mapsto u.\textrm{Id}$. Using
the identification (\ref{dep}), we have $ $$\lambda=\sum_{i>0}\mu_{i}+2\mu_{0}$.
The associated character of $\widehat{T}$ is $\sum_{i>0}\chi_{i}+2\chi_{0}$
, which is the function $(A,x_{0})\mapsto\textrm{det}(A)x_{0}^{2}$.
\end{proof}
Let $g=(A,x_{0})\in\widehat{T}$, with $A=\textrm{diag}(x_{1},...,x_{n})$.
Then $r(g(\sigma\cdot g))=\textrm{det}(A)x_{0}^{2}\textrm{diag}(\frac{x_{n}}{x_{1}},...,\frac{x_{1}}{x_{n}})$,
and:

\begin{align*}
H_{p}(t) & =\textrm{det}\left(t-p^{n-1}r(g(\sigma\cdot g))\right)\\
 & =\prod_{i=1}^{n}\left(t-p^{n-1}\textrm{det}(A)x_{0}^{2}\frac{x_{n+1-i}}{x_{i}}\right)\\
 & =R(t)\times(t-p^{n-1}\textrm{det}(A)x_{0}^{2})
\end{align*}
where $R(t)={\displaystyle \prod_{i\neq k}}\left(t-p^{n-1}\textrm{det}(A)x^{2}\frac{x_{n+1-i}}{x_{i}}\right)$.
The polynomial $R$ is invariant under $\Omega(T)$ and twisted conjugation.
We deduce the following result.
\begin{thm}
\label{facteurH}The Hecke polynomial $H_{p}$ in $\mathcal{H}(G(\mathbb{Q}_{p})//K_{p})$
factors into a product
\[
H_{p}(t)=R(t).(t-p^{n-1}1_{pK_{p}})
\]
where $R(t)\in\mathcal{H}(G(\mathbb{Q}_{p})//K_{p})[t]$.
\end{thm}

\section{The Shimura variety}

\subsection{The moduli problem}

Let $K^{p}\subset G(\mathbb{A}_{f}^{p})$ be a compact open subgroup
and denote by $Sh_{K}$ the moduli space associated to the data $(E,\sigma,V,\psi,\mathcal{O}_{E,(p)},\Lambda,h,\mu)$
according to Kottwitz \cite{Kot}. We assume $K^{p}$ to be sufficiently
small such that this moduli problem is representable by a smooth,
quasi-projective scheme over $\mathcal{O}_{E_{p}}$. For any noetherian
$\mathcal{O}_{E_{p}}$-scheme $S$, it classifies the following data,
up to prime-to-$p$-isogeny:
\begin{enumerate}
\item An abelian scheme $A$ of dimension $n$ over $S$.
\item A $\mathbb{Q}$-homogeneous polarization $\overline{\lambda}=\mathbb{Q}\lambda$
for some prime-to-$p$ polarization $\lambda$.
\item An action $\iota:\mathcal{O}_{E}\otimes\mathbb{Z}_{(p)}\hookrightarrow\textrm{End}(A)\otimes\mathbb{Z}_{(p)}$
compatible with $\overline{\lambda}$.
\item A $\pi_{1}(S,s)$-stable $K^{p}$-orbit of compatible isomorphisms
$\overline{\eta}:V(\mathbb{A}_{f}^{p})\xrightarrow{\sim}H_{1}(A_{s},\mathbb{A}_{f}^{p})$
for one geometric point $s$ in each connected component of $S$.
\end{enumerate}
Further, $(A,\iota,\overline{\lambda},\overline{\eta})$ satisfies
the determinant condition: the characteristic polynomial of $e\in\mathcal{O}_{E}\otimes\mathbb{Z}_{(p)}$
acting on $\textrm{Lie}(A)$ is $(T-e)^{n-1}(T-\overline{e})\in\mathcal{O}_{S}[T]$.

We now give an equivalent moduli problem. Write $\widehat{\mathbb{Z}}^{(p)}={\displaystyle \prod_{\ell\neq p}}\mathbb{Z}_{\ell}\subset\mathbb{A}_{f}^{p}$
and for any $\mathcal{O}_{E}\left[\frac{1}{p}\right]$-lattice $L\subset V$,
write $\widehat{L}^{(p)}=L\otimes\widehat{\mathbb{Z}}^{(p)}\subset V(\mathbb{A}_{f}^{p})$.
We can find a $\mathcal{O}_{E}\left[\frac{1}{p}\right]$-lattice $L\subset V$
satisfying the conditions:

\begin{equation}
\begin{array}{c}
K^{p}\subset\left\{ g\in G(\mathbb{A}_{f}^{p}),\: g(\widehat{L}^{(p)})=\widehat{L}^{(p)}\right\} \\
\varphi(L,L)\subset\mathbb{Z}\left[\frac{1}{p}\right]
\end{array}\label{condL}
\end{equation}
(see notations for the definition of $\varphi$). The determinant
of $\varphi:L\times L\rightarrow\mathbb{Z}\left[\frac{1}{p}\right]$
is a square in $\mathbb{Z}\left[\frac{1}{p}\right]$ and is well defined
up to an invertible element. Let $d\in\mathbb{Z}$ coprime to $p$
such that $\textrm{det}(\varphi)=d^{2}$. We consider the moduli problem
$\mathfrak{F}$ classifying the following data, up to isomorphism:
For any noetherian $\mathcal{O}_{E_{p}}$-scheme $S$,
\begin{enumerate}
\item An abelian scheme $A$ of dimension $n$ over $S$.
\item A polarization $\lambda:A\rightarrow A^{\vee}$ of degree $d^{2}$.
\item An action $\iota:\mathcal{O}_{E}\hookrightarrow End(A)$ compatible
with $\lambda$.
\item A $\pi_{1}(S,s)$-stable $K^{p}$-orbit of compatible isomorphisms$\overline{\eta}:\widehat{L}^{(p)}\xrightarrow{\sim}T_{f}^{p}(A_{s})=\prod_{\ell\neq p}T_{\ell}(A_{s})$
such that the following diagram commutes
\end{enumerate}
\[
\begin{array}{ccccc}
\widehat{L}^{(p)} & \times & \widehat{L}^{(p)} & \longrightarrow & \widehat{\mathbb{Z}}^{(p)}\\
\eta^{p}\downarrow &  & \eta^{p}\downarrow &  & \downarrow\theta\\
T_{f}^{p}(A_{s}) & \times & T_{f}^{p}(A_{s}) & \longrightarrow & \widehat{\mathbb{Z}}^{(p)}(1)
\end{array}
\]
where $\theta$ is some $\widehat{\mathbb{Z}}^{(p)}$-linear isomorphism.
Further, $(A,\iota,\lambda,\overline{\eta})$ satisfies the determinant
condition.

Proposition \ref{shkisom} below is well known, we will skip the proof.
\begin{prop}
\textup{\label{shkisom}The natural map $\mathcal{\mathfrak{F}}\rightarrow Sh_{K}$
is an isomorphism of functors.}\end{prop}
\begin{rem}
\label{isomtriv}For $K^{p}$ small enough, every automorphism of
a tuple $(A,\iota,\lambda,\overline{\eta})$ is trivial.
\end{rem}

\subsection{Dieudonné modules}

We write $\overline{Sh_{K}}=Sh_{K}\times\kappa(\mathcal{O}_{E_{p}})$
for the special fibre of $Sh_{K}$. It is equidimensional of dimension
$n-1$. In order to study $\overline{Sh_{K}}$, we use (covariant)
Dieudonné theory.

\subsubsection{Some definitions}

Let $k'$ be an algebraically closed field containing $\kappa(\mathcal{O}_{E_{p}})$
and let $W=W(k')$ be the ring of Witt vectors and $W_{\mathbb{Q}}=W\otimes\mathbb{Q}$.
The choice of $k$ induces an embedding $\varrho:E_{p}\hookrightarrow W_{\mathbb{Q}}$.
A \textsl{Dieudonné module} is a free $W$-module $M$ of finite rank
together with a $\sigma$-linear endomorphism $F$ and a $\sigma^{-1}$-endomorphism
$V$ of $M$ such that $FV=VF=p$.

A \textsl{Dieudonné space} over $k$ is a finite-dimensional $k$-vector
space together with a $\textrm{Frob}_{k}$-linear endomorphism and
a $\textrm{Frob}_{k}^{-1}$-endomorphism $V$ of $M$ such that $FV=VF=0$.
If $M$ is a Dieudonné module, then $\overline{M}=\frac{M}{pM}$ is
a Dieudonné space which satisfies
\begin{equation}
\textrm{Im}(F)=\textrm{Ker}(V)\;\textrm{ and }\;\textrm{Im}(V)=\textrm{Ker}(F).\label{egalitesbt1}
\end{equation}

An $\mathcal{O}_{E_{p}}$-\textsl{Dieudonné module} \textsl{over}
$k$ is a Dieudonné module endowed with a $W$-linear $\mathcal{O}_{E_{p}}$-action
commuting with $F,V$. We define similarly the notion of $\mathcal{O}_{E_{p}}$-\textsl{Dieudonné
space.} The $\mathcal{O}_{E_{p}}$-action induces a decomposition
$M=M_{e}\oplus M_{\overline{e}}$ where $M_{e}$ (resp. $M_{\overline{e}}$)
is the submodule where $\mathcal{O}_{E_{p}}$ acts via $\varrho$
(resp. $\overline{\varrho}$). We define the \textsl{signature} of
an $\mathcal{O}_{E_{p}}$-Dieudonné module to be the pair

\[
\left(\textrm{dim}_{k}(\frac{M_{e}}{VM_{\overline{e}}}),\textrm{dim}_{k}(\frac{M_{\overline{e}}}{VM_{e}})\right).
\]

If $A$ is an abelian variety over $k$ and $M=\mathbb{D}(A)$, then
$\textrm{Lie}(A)=\frac{M}{VM}$. We define in a similar fashion the
signature of an $\mathcal{O}_{E_{p}}$-Dieudonné space. The signatures
of $M$ and $\overline{M}$ are the same.

A \textsl{quasi-unitary Dieudonné module} \textsl{over} $k$ is an
$\mathcal{O}_{E_{p}}$-Dieudonné module endowed with a non-degenerate
alternating pairing $\left\langle ,\right\rangle :M\times M\rightarrow W_{\mathbb{Q}}$
such that $\forall e\in\mathcal{O}_{E_{p}}$, $\forall x,y\in M$,
$\left\langle ex,y\right\rangle =\left\langle x,\overline{e}y\right\rangle $
and $\left\langle Fx,y\right\rangle =\sigma\left\langle x,Vy\right\rangle $.
We call $M$ \textsl{unitary} if $\left\langle ,\right\rangle :M\times M\rightarrow W$
is perfect. We define similarly the notion of \textsl{unitary Dieudonné
space.} 

A \textsl{unitary isocrystal over} $k$ is a finite-dimensional $W_{\mathbb{Q}}$-vector
space $N$ together with endomorphisms $F,V$, an $\mathcal{O}_{E_{p}}$-action,
a $W_{\mathbb{Q}}$-bilinear pairing $\left\langle ,\right\rangle :N\times N\rightarrow W_{\mathbb{Q}}$
subject to the same hypotheses as above. If $M$ is a quasi-unitary
Dieudonné module, then $M\otimes W_{\mathbb{Q}}$ is a unitary isocrystal.
If $\lambda\in\mathbb{Q}$, we denote by $N_{\lambda}$ ``the''
simple isocrystal of slope $\lambda$. We say that an isocrystal is
supersingular if all its slopes are $\frac{1}{2}$.

\subsubsection{Dieudonné theory}

Dieudonné theory gives an equivalence of categories between unitary
Dieudonné modules over $k$ and $p$-divisible groups over $k$ (with
polarization and $\mathcal{O}_{E_{p}}$-action). For a definition
of these objects, see \cite{BW1} section 2. Similarly, there is an
equivalence of categories between unitary Dieudonné spaces over $k$
which satisfy (\ref{egalitesbt1}) and truncated Barsotti-Tate groups
of level 1 (or $BT_{1}$) over $k$ (with polarization and $\mathcal{O}_{E_{p}}$-action).
See \cite{grothBT} definition 3.2 for these statements.

\subsubsection{Examples}
\begin{enumerate}
\item \noindent Let $SS$ be the following Dieudonné module: It has a $W$-basis
$(g,h)$ such that $SS_{e}=Wg$, $SS_{\overline{e}}=Wh$, the endomorphisms
$F,V$ are defined by $F(g)=h=-V(g)$, and the pairing is given by
$\left\langle g,h\right\rangle =1$. This is a unitary Dieudonné module
of signature $(1,0)$ and slope $\frac{1}{2}$.
\item Let $d\geq1$ be an integer. Define a unitary Dieudonné module $\mathbb{B}(d)$
as follows: It has a $W$-basis $(e_{i},f_{i})$, $i\in\{1,...,d\}$
with $e_{i}\in\mathbb{B}(d)_{e}$ and $f_{i}\in\mathbb{B}(d)_{\overline{e}}$.
The endomorphisms $F,V$ are given by
\begin{align*}
 & F(f_{1})=(-1)^{d}e_{n}\\
 & F(e_{i})=f_{i-1}\textrm{ for }i=2,...,d\\
 & V(f_{d})=e_{1}\\
 & V(e_{i})=f_{i+1}\textrm{ for }i=1,...,d-1.
\end{align*}
The alternating form is defined by $\left\langle e_{i},f_{j}\right\rangle =(-1)^{i-1}\delta_{i,j}$.
This is a unitary Dieudonné module of signature $(d-1,1)$. If $d$
is odd, every slope of $\mathbb{B}(d)\otimes W_{\mathbb{Q}}$ is $\frac{1}{2}$.
If $d$ is even, its slopes are $\frac{1}{2}\pm\frac{1}{d}$ (cf.
\cite{BW1} lemma 3.3).
\end{enumerate}

\subsubsection{Classification}

We classify below isocrystals and Dieudonné spaces which come into
play in our situation. We refer to \cite{BW1}, section 3.1 and 3.6
respectively for the proofs.
\begin{prop}
\label{isoc}Let $M$ be a unitary Dieudonné module of signature $(n-1,1)$
and $N$ its isocrystal. Then
\[
N\simeq N(r)\times(N_{\frac{1}{2}})^{n-2r}
\]
where $r$ is an integer $0\leq r\leq\frac{n-1}{2}$ and:
\[
N(r)=\begin{cases}
0 & \textrm{if }r=0\\
N_{\frac{1}{2}-\frac{1}{2r}}\oplus N_{\frac{1}{2}+\frac{1}{2r}} & \textrm{if }r>0\textrm{ is even}\\
N_{\frac{1}{2}-\frac{1}{2r}}^{2}\oplus N_{\frac{1}{2}+\frac{1}{2r}}^{2} & \textrm{if }r\textrm{ is odd}.
\end{cases}
\]

\end{prop}
\vspace{0bp}

\begin{prop}
\label{BT1}Let $\overline{M}$ be a unitary Dieudonné space of signature
$(n-1,1)$. There is an integer $1\leq r\leq n$ such that $\overline{M}$
is isomorphic to
\[
\overline{\mathbb{B}(r)}\oplus\overline{SS}^{n-r}.
\]

\end{prop}

\subsection{Stratifications}

\subsubsection{Ekedahl-Oort stratification}

Applying proposition \ref{BT1} to $M$ gives us an integer $1\leq r\leq n$.
This defines a stratification
\[
\overline{Sh_{K}}=\bigsqcup_{r=1}^{n}\mathcal{M}_{r}
\]
where $\mathcal{M}_{r}$ is the locus where $\overline{M}$ is isomorphic
to $\overline{\mathbb{B}(r)}\oplus\overline{SS}^{n-r}$. A point in
$\mathcal{M}_{r}$ and its Dieudonné module are said of type $r$.
All the $\mathcal{M}_{r}$ are equidimensional and the dimensions
are given by: 
\begin{eqnarray*}
\textrm{dim}(\mathcal{M}_{2i}) & = & n-i\\
\textrm{dim}(\mathcal{M}_{2i+1}) & = & i
\end{eqnarray*}
(cf. \textsl{ibid}, 5.4).

\subsubsection{Newton polygon stratification}

The Newton polygon stratification is given by isomorphism classes
of unitary isocrystals. It happens to be coarser than the Ekedahl-Oort
stratification. It reads:
\[
\overline{Sh_{K}}=\mathcal{M}_{2}\sqcup\mathcal{M}_{4}\sqcup...\cup\mathcal{M}_{n-1}\sqcup\bigsqcup_{r\textrm{ odd}}\mathcal{M}_{r}.
\]
The stratum $\mathcal{M}_{2r}$ is also the locus where the unitary
isocrystal is isomorphic to $N(r)\times(N_{\frac{1}{2}})^{n-2r}$.
The supersingular locus is $\overline{Sh_{K}}^{ss}={\displaystyle \bigsqcup_{r\textrm{ odd}}}\mathcal{M}_{r}$
and has dimension $\frac{n-1}{2}$ (\textsl{ibid}, proposition 5.5).
Finally, we state a result on the geometric structure of $\overline{Sh_{K}}^{ss}$.
For the proof, see \cite{V2}, theorem 5.2.
\begin{thm}
For $K^{p}$ sufficiently small, the supersingular locus $\overline{Sh_{K}}^{ss}$
is equidimensional of dimension $\frac{n-1}{2}$ and locally of complete
intersection. Its smooth locus is the open Ekedahl-Oort stratum $\mathcal{M}_{n}$.
\end{thm}

\section{Moduli space of $p$-isogenies}

\subsection{The moduli problem}

We define a moduli space classifying $p$-isogenies. Let $S$ be an
$\mathcal{O}_{E_{p}}$-scheme and $\underline{A}_{i}=(A_{i},\iota_{i},\overline{\lambda}_{i},\overline{\eta}_{i})$,
$i\in\left\{ 1,2\right\} $ two tuples corresponding to $S$-valued
points of $Sh_{K}$. A $p$-\textsl{isogeny} $f:\underline{A}_{1}\rightarrow\underline{A}_{2}$
is an $\mathcal{O}_{E,(p)}$-linear isogeny such that $p^{c}\lambda_{1}=f^{\vee}\circ\lambda_{0}\circ f$
for some $c\geq0$, which we call the multiplicator. This implies
$\textrm{deg}(f)=p^{cn}$.

Let $p\textrm{-}\mathscr{I}sog$ be the $\mathcal{O}_{E_{p}}$-scheme
classifying $p$-isogenies. Two $p$-isogenies $f:\underline{A}_{1}\rightarrow\underline{A}_{2}$
and $f':\underline{A}'_{1}\rightarrow\underline{A}'_{2}$ are identified
if there are prime-to-$p$-isogenies $h_{i}:\underline{A}_{i}\rightarrow\underline{A}'_{i}$
for $i\in\left\{ 1,2\right\} $ such that $f'\circ h_{1}=h_{2}\circ f$.
The $p$-isogenies of multiplicator $c$ form an open and closed subscheme
$p\textrm{-}\mathscr{I}sog^{(c)}\subset p\textrm{-}\mathscr{I}sog$.

Let $S$ be an $\mathcal{O}_{E_{p}}$-scheme. Write $\mathfrak{I}(S)$
for the moduli problem classifying $p$-isogenies between points of
$\mathfrak{F}(S)$ (see 2.1 for the definition of $\mathfrak{F}$),
up to isomorphisms. The natural map $\mathfrak{I}\rightarrow p\textrm{-}\mathscr{I}sog$
is an isomorphism of functors.

Let $s,t:p\textrm{-}\mathscr{I}sog\rightarrow Sh_{K}$ be the maps
sending an isogeny to its source and target, respectively. The restrictions
$s,t:p\textrm{-}\mathscr{I}sog^{(c)}\rightarrow Sh_{K}$ are proper
for $c\geq0$ (\cite{BM}, 4.2.1).

The ``multiplication by $p$'' map sends $\underline{A}$ to $p:\underline{A}\rightarrow\left\langle p\right\rangle \underline{A}$
(for $q\in p^{\mathbb{Z}}$,$\left\langle q\right\rangle $ multiplies
the level structure by $q$). This defines a section of $s$. As $s$
is separated, its image is a reduced closed subscheme $\left\langle p\right\rangle \subset p\textrm{-}\mathscr{I}sog^{(2)}$.

The ordinary locus $p\textrm{-}\mathscr{I}sog^{ord}\times\kappa(\mathcal{O}_{E_{p}})$
is defined as the inverse image of $\overline{Sh_{K}}^{ord}$ by$s$
(or $t$). We define similarly the supersingular locus $p\textrm{-}\mathscr{I}sog^{ss}\times\kappa(\mathcal{O}_{E_{p}})$.
On the special fibre, there is a Frobenius section of $s$. It sends
a tuple $\underline{A}$ to the Frobenius isogeny $F_{\underline{A}}:\underline{A}\rightarrow\underline{A}^{(p^{2})}$.
The level structure on $\overline{\eta}^{(p^{2})}$ on $A^{(p^{2})}$
is compatible with $\overline{\eta}$ through $F_{A}$. The image
of $s$ is a reduced closed subscheme $F\subset p\textrm{-}\mathscr{I}sog\times\kappa(\mathcal{O}_{E_{p}})$,
which is a union of irreducible components of $p\textrm{-}\mathscr{I}sog\times\kappa(\mathcal{O}_{E_{p}})$,
because $s$ is finite and flat over the ordinary locus (\cite{BM},
4.2.2) and $\overline{Sh_{K}}^{ord}$ is dense in $\overline{Sh_{K}}$.
This follows also from the fact that $p\textrm{-}\mathscr{I}sog\times\kappa(\mathcal{O}_{E_{p}})$
is equidimensional of dimension $n-1$, as we will show later. By
duality, we also have the Verschiebung map $V_{A}:A^{(p^{2})}\rightarrow A$.
Notice that $V_{A}\circ F_{A}=p^{2}$ so taking into account level
structures, the Verschiebung is a map $V_{\underline{A}}:\left\langle p^{-2}\right\rangle \underline{A}^{(p^{2})}\rightarrow\underline{A}$.

\subsection{The $\mathbb{Q}$-algebra $\mathbb{Q}\left[p\textrm{-}\mathscr{I}sog\times L\right]$}

Composition of isogenies defines a morphism

\[
c:p\textrm{-}\mathscr{I}sog\times_{t,s}p\textrm{-}\mathscr{I}sog\longrightarrow p\textrm{-}\mathscr{I}sog
\]
which is proper (cf. \cite{BM}, 4.2.1). Let $L$ be a field and $\mathcal{O}_{E_{p}}\rightarrow L$
a homomorphism. Let $Z_{\mathbb{Q}}\left(p\textrm{-}\mathscr{I}sog\times L\right)$
denote the group of algebraic cycles of $p\textrm{-}\mathscr{I}sog\times L$,
with $\mathbb{Q}$-coefficients. For cycles $Y_{1},Y_{2}$, we define:
\[
Y_{1}\cdot Y_{2}=c_{*}\left(Y_{1}\times_{t,s}Y_{2}\right).
\]
Extending this product bilinearly, we get a ring structure on $Z_{\mathbb{Q}}\left(p\textrm{-}\mathscr{I}sog\times L\right)$,
with identity $p\textrm{-}\mathscr{I}sog^{(0)}\times L$. Let $\mathbb{Q}\left[p\textrm{-}\mathscr{I}sog\times L\right]$
be the $\mathbb{Q}$-subalgebra generated by the irreducible components.

Define the $\mathbb{Q}$-algebra $\mathbb{Q}\left[p\textrm{-}\mathscr{I}sog^{ord}\times\kappa(\mathcal{O}_{E_{p}})\right]$
in a similar fashion as hereabove. We may view $F$ as an element
of $\mathbb{Q}\left[p\textrm{-}\mathscr{I}sog^{ord}\times\kappa(\mathcal{O}_{E_{p}})\right]$
or $\mathbb{Q}\left[p\textrm{-}\mathscr{I}sog\times\kappa(\mathcal{O}_{E_{p}})\right]$.

\subsection{A commutative diagram}

Let $\mathcal{H}_{0}(G(\mathbb{Q}_{p})//K_{p})\subset\mathcal{H}(G(\mathbb{Q}_{p})//K_{p})$
be the subalgebra of $\mathbb{Q}$-valued functions that have support
contained in $G(\mathbb{Q}_{p})\cap\textrm{End}(\Lambda)$. There
is a $\mathbb{Q}$-algebra homomorphism

\[
h:\mathcal{H}_{0}(G(\mathbb{Q}_{p})//K_{p})\longrightarrow\mathbb{Q}\left[p\textrm{-}\mathscr{I}sog\times E_{p}\right]
\]
which we will explain briefly. Let $L$ be a field containing $E_{p}$
and $f:\underline{A}_{1}\rightarrow\underline{A}_{2}$ corresponding
to an $L$-valued point in $p\textrm{-}\mathscr{I}sog\times E_{p}$.
Choose isomorphisms $\alpha_{i}:\Lambda\simeq T_{p}(A_{i})$ ,$i\in\left\{ 0,1\right\} $.
Then $\alpha_{2}^{-1}\circ V_{p}f\circ\alpha_{1}:\Lambda\otimes\mathbb{Q}_{p}\rightarrow\Lambda\otimes\mathbb{Q}_{p}$
is an element of $G(\mathbb{Q}_{p})\cap\textrm{End}(\Lambda$). Its
class $\tau(f)$ in $K_{p}\backslash G(\mathbb{Q}_{p})/K_{p}$ is
independent of the choices involved. The function $\tau$ is constant
on irreducible components of $p\textrm{-}\mathscr{I}sog\times E_{p}$.
Then $h$ maps $1_{K_{p}gK_{p}}$ to the sum of irreducible components
$C\subset p\textrm{-}\mathscr{I}sog\times E_{p}$ such that $\tau(C)=K_{p}gK_{p}$.
The specialization map
\[
\sigma:\mathbb{Q}[p\textrm{-}\mathscr{I}sog\times E_{p}]\longrightarrow\mathbb{Q}[p\textrm{-}\mathscr{I}sog\times\kappa(\mathcal{O}_{E_{p}})]
\]
is defined as follows: Let $C$ an irreducible composant of $p\textrm{-}\mathscr{I}sog\times E_{p}$
and $\mathcal{C}$ the scheme-theoretic image of $C$ by the open
immersion $p\textrm{-}\mathscr{I}sog\times E_{p}\hookrightarrow p\textrm{-}\mathscr{I}sog$.
Then $\sigma(C)=\left[\mathcal{C}\times\kappa(\mathcal{O}_{E_{p}})\right]$.

The cocharacter $\mu$ is defined over $E$. Let $M\subset G_{\mathbb{Q}_{p}}$
be the centralizer of $\mu\overline{\mu}$. We denote again by $\mathcal{H}_{0}(M(\mathbb{Q}_{p})//(K_{p}\cap M(\mathbb{Q}_{p})))$
the functions with support in $\Lambda$. We have a commutative diagram
of $\mathbb{Q}$-algebra homomorphisms

\begin{equation}
\xymatrix{\mathcal{H}_{0}(G(\mathbb{Q}_{p})//K_{p})\ar[dd]_{\dot{S}_{M}^{G}}\ar[r]^{h} & \mathbb{Q}\left[p\textrm{-}\mathscr{I}sog\times E_{p}\right]\ar[d]^{\sigma}\\
 & \mathbb{Q}\left[p\textrm{-}\mathscr{I}sog\times\kappa(\mathcal{O}_{E_{p}})\right]\ar@<1ex>[d]^{\textrm{ord}}\\
\mathcal{H}_{0}(M(\mathbb{Q}_{p})//(K_{p}\cap M(\mathbb{Q}_{p})))\ar[r]_{h_{0}} & \mathbb{Q}\left[p\textrm{-}\mathscr{I}sog^{ord}\times\kappa(\mathcal{O}_{E_{p}})\right]\ar@<1ex>[u]^{\textrm{cl}}
}
\label{diagcommutatif}
\end{equation}
The morphism $\dot{S}$ is the twisted Satake homomorphism (see \cite{W1}
$\mathsection$1). The map $\textrm{ord}$ is defined by intersecting
with the ordinary locus, and the map $\textrm{cl}$ is defined by
taking the closure of a cycle of $p\textrm{-}\mathscr{I}sog^{ord}\times\kappa(\mathcal{O}_{E_{p}})$.
The big square is commutative. In this context, the ``congruence
relation'' means the following conjecture:
\begin{conjecture*}
\label{rdcconjpisog-1}Consider the polynomial $H_{\mathfrak{p}}$
inside $\mathbb{Q}\left[p\textrm{-}\mathscr{I}sog\times\kappa(\mathcal{O}_{E_{p}})\right]$
via the morphism $\sigma\circ h^{0}$. The element $F$ lies in the
center of this ring and the following relation holds:\textup{
\[
H_{\mathfrak{p}}(F)=0.
\]
}
\end{conjecture*}
This relation makes sense since $F$ belongs to the center of $\mathbb{Q}\left[p\textrm{-}\mathscr{I}sog\times\kappa(\mathcal{O}_{E_{p}})\right]$,
as we shall see in section 5. We mention the following theorem, due
to B. Moonen in \cite{BM}.
\begin{thm*}
Consider the polynomial $H_{\mathfrak{p}}$ inside $\mathbb{Q}\left[p\textrm{-}\mathscr{I}sog^{ord}\times\kappa(\mathcal{O}_{E_{p}})\right]$
via the morphism \textup{$\textrm{ord}\circ\sigma\circ h^{0}$}. In
this ring, the following relation holds:\textup{
\[
H_{\mathfrak{p}}(F)=0.
\]
}
\end{thm*}

\section{The supersingular locus}

\subsection{The moduli space $\mathcal{N}'$}

Uniformization theory from \cite{RZ} can be used in order to study
the supersingular locus of $\overline{Sh_{K}}$. In \cite{V1,V2},
the authors give the geometric structure of $\overline{Sh_{K}}^{ss}$.
We state below their main results.

Let $K^{p}\subset G(\mathbb{A}_{f}^{p})$ be an open compact subgroup.
We fix a tuple $\underline{A}'=(A',\iota',\overline{\lambda}',\overline{\eta}')$
over $\overline{\mathbb{F}}$. Using the same conventions as in \cite{RZ},
$\overline{\eta}'$ is a $K^{p}$-orbit of isomorphisms: $\overline{\eta}':H_{1}(A',\mathbb{A}_{f}^{p})\rightarrow V(\mathbb{A}_{f}^{p})$.
We assume that $A'$ is supersingular. We denote by $\underline{X}'$
its $p$-divisible group over $\overline{\mathbb{F}}$ and we write
$M'=\mathbb{D}(A')$ and $N'=M'\otimes W_{\mathbb{Q}}$.

The formal scheme $\mathcal{N}'$ over $\overline{\mathbb{F}}$ classifies
the following pairs $(\underline{X},\rho_{X})$ up to prime-to-$p$-isogenies.
Given a $\overline{\mathbb{F}}$-scheme $S$, $\underline{X}$ is
a $p$-divisible group with unitary structure of signature $(n-1,1)$
over $S$ and $\rho_{X}:\underline{X}\longrightarrow\underline{X}'_{S}$
is a quasi-isogeny such that $\rho_{X}^{*}(\overline{\lambda}')=p^{c}\overline{\lambda}$
for some $c\in\mathbb{Z}$.

Dieudonné theory gives a bijection between $\mathcal{N}'(\overline{\mathbb{F}})$
and the set of quasi-unitary Dieudonné modules $M\subset N'$ of signature
$(n-1,1)$ such that $p^{c}M^{\vee}=M$ for some $c\in\mathbb{Z}$.
If $(\underline{X},\rho_{X})\in\mathcal{N}'(S)$, there is a unique
tuple $\underline{A}=(A,\iota,\overline{\lambda},\overline{\eta})$
over $S$ with a quasi-isogeny $f:\underline{A}\rightarrow\underline{A}'$
lifting $\rho_{X}$. We write $\underline{A}=\rho_{X}^{*}\underline{A}'$.

If $g\in G(\mathbb{A}_{f})$, and $\underline{A}=(A,\iota,\overline{\lambda},\overline{\eta})$
is a tuple over $S$, we define $\left\langle g\right\rangle \underline{A}=(A,\iota,\overline{\lambda},\overline{g\circ\eta})$.
The uniformization morphism is given by: 
\[
\begin{array}{ccccc}
\Theta & : & \mathcal{N}'\times G(\mathbb{A}_{f}^{p}) & \longrightarrow & \overline{Sh_{K}}^{ss}\times\overline{\mathbb{F}}\\
 &  & (X,\rho_{X})\times g & \longmapsto & \left\langle g\right\rangle \rho_{X}^{*}\underline{A}'
\end{array}.
\]

Let $I$ be the algebraic group over $\mathbb{Q}$ of $\mathcal{O}_{E,(p)}$-linear
quasi-isogenies in $\textrm{End}^{0}(A')$ compatible with $\overline{\lambda}'$.
We have a natural homomorphism $\alpha_{p}:I(\mathbb{Q}_{p})\hookrightarrow J(\mathbb{Q}_{p})$,
where $J$ denotes the $\mathbb{Q}_{p}$-algebraic group of automorphisms
of $N'$ respecting the polarization up to factor. An element $\eta'\in\overline{\eta'}$
provides a homomorphism $\alpha^{p}:I(\mathbb{Q})\rightarrow J(\mathbb{A}_{f}^{p})$
(for more details, see \cite{RZ}, 6.15). We have the following theorem
(\cite{RZ} theorem 6.30):
\begin{thm}
The uniformization theorem induces an isomorphism of $\overline{\mathbb{F}}$-schemes:
\[
I(\mathbb{Q})\diagdown\mathcal{N}{}_{red}'\times G(\mathbb{A}_{f}^{p})\diagup K^{p}\longrightarrow\overline{Sh_{K}}^{ss}\times\overline{\mathbb{F}}.
\]

\end{thm}
\begin{flushleft}
Write $I(\mathbb{Q})\diagdown G(\mathbb{A}_{f}^{p})\diagup K^{p}=\left\{ g_{1},...,g_{m}\right\} $
and $\Gamma_{j}=I(\mathbb{Q})\cap g_{j}K^{p}g_{j}^{-1}$. There is
a decomposition: 
\[
I(\mathbb{Q})\diagdown\mathcal{N}_{red}'\times G(\mathbb{A}_{f}^{p})\diagup K^{p}=\coprod_{j=1}^{m}\Gamma_{j}\diagdown\mathcal{N}_{red}'.
\]
We now recall some results from \cite{V1,V2}. Notice that in these
articles the signature of $A'$ is $(1,n-1)$. That's why we modify
slightly the definition of $\mathcal{L}_{i}(n)$ (the integer $i$
is replaced by $i-1$). The scheme $\mathcal{N}_{red}'$ has a stratification
\[
\mathcal{N}_{red}'=\bigcup_{i\in2\mathbb{Z}}\mathcal{N}'_{red,i}
\]
where $\mathcal{N}'_{red,i}$ is the open and closed subscheme of
elements of multiplicator $i$. Observe that $\mathcal{N}'_{red,i}$
is empty if $i$ is odd (\cite{V2} (1.5.1)). For $i$ even, all the
$\mathcal{N}'_{red,i}$ are isomorphic to one another \textsl{(ibid},
proposition 1.1). Write $\mathbf{N}'_{0}=\left\{ x\in N'_{e},\tau x=x\right\} $,
where $\tau=p^{-1}F^{2}$. This is a $\mathbb{Q}_{p^{2}}$-hermitian
space for the form $\left\{ x,y\right\} =\alpha\left\langle x,Fy\right\rangle $.
Define:
\begin{equation}
\mathcal{L}_{i}(n)=\left\{ L\subset\mathbf{N}'_{0},\mathbb{Z}_{p^{2}}\textrm{-lattice},L=p^{i-1}L^{\wedge}\right\} \label{Li}
\end{equation}
where $L^{\wedge}$ is the dual lattice with respect to $\left\{ ,\right\} $.
For each $L\in\mathcal{L}_{i}(n)$, there is an associated closed
subscheme $\mathcal{N}'_{L}\subset\mathcal{N}'_{red,i}$. We have
the following decomposition in irreducible components:
\[
\mathcal{N}_{red,i}'=\bigcup_{L\in\mathcal{L}_{i}(n)}\mathcal{N}'_{L}
\]
(\textsl{ibid}, theorem 4.2). The $\mathcal{N}'_{L}$ are all isomorphic
for $L\in\mathcal{L}_{i}(n)$, smooth, of dimension $\frac{n-1}{2}$.
We say that a point $(\underline{X},\rho_{X})\in\mathcal{N}_{red}'$
has type $r$ if its Dieudonné module has type $r$ (see 2.4.1). The
smooth locus of $\mathcal{N}_{red}'$ is the set of points of type
$n$.
\par\end{flushleft}

Further, there is a bijection between quasi-unitary superspecial Dieudonné
modules $M\subset N'$ of signature $(n,0)$ such that $p^{i-1}M^{\vee}=M$
and lattices in $\mathcal{L}_{i}(n)$. The bijection is given by $M\mapsto M_{e}^{\tau}$.
If $L\in\mathcal{L}_{i}(n)$, write $L^{+}$ for the associated superspecial
Dieudonné module of signature $(n,0)$. We thus get a bijection between
irreducible components of $\mathcal{N}_{red,i}'$ and quasi-unitary
superspecial Dieudonné modules of signature $(n,0)$ which satisfy
$p^{i-1}M^{\vee}=M$. If $y$ a point in $\mathcal{N}'_{red,i}(\overline{\mathbb{F}})$
with Dieudonné module $M$, then $y$ lies in $\mathcal{N}'_{L}(\overline{\mathbb{F}})$
if and only if $M\subset L^{+}$ (\textsl{ibid}, lemma 3.3). If $y\in\mathcal{N}'_{L}(\overline{\mathbb{F}})$
has type $n$, then $L^{+}=\Lambda^{+}(M)$, the smaller superspecial
Dieudonné module containing $M$.

\subsection{Dimension of the fibers of $s,t$}

Let $c\geq0$ be a fixed even integer and $x$ be an $\overline{\mathbb{F}}$-valued
point of $\overline{Sh_{K}}^{ss}$, corresponding to a tuple $\underline{A}'=(A',\iota',\overline{\lambda}',\overline{\eta}')$
over $\overline{\mathbb{F}}$. Let $M'$ be its Dieudonné module and
$N'$ its isocrystal. We write $t_{c}^{-1}(x)$ for the fibre of $t$
above $x$ in $p\textrm{-}\mathscr{I}sog^{(c)}\times\overline{\mathbb{F}}$.
We consider the moduli space $\mathcal{N}_{red}'$ associated to $\underline{A}'$,
as above. We assume that $K^{p}$ satisfies the condition of remark
\ref{isomtriv}. Then, there is a well-defined morphism of $\overline{\mathbb{F}}$-schemes:
\[
\epsilon:t_{c}^{-1}(x)\longrightarrow\mathcal{N}_{red}'
\]
sending an isogeny $f:\underline{A}\rightarrow\underline{A}'$ on
the induced isogeny $f:\underline{X}\rightarrow\underline{X}'$ on
the $p$-divisibles groups (forgetting the level structure). It can
be showed that $\epsilon$ is proper using the valuative criterion.
Further, $\epsilon$ is injective on $S$-points for all $\overline{\mathbb{F}}$-scheme
$S$, since $\underline{A}$ can be reconstructed from $f:\underline{X}\rightarrow\underline{X}'$.
Thus $\epsilon$ is a closed immersion (\cite{EGA4}, 8.11.5).

The $\overline{\mathbb{F}}$-points of $t_{c}^{-1}(x)$ are in bijection
with the quasi-unitary Dieudonné $M$ modules over $k$ of signature
$(n-1,1)$ satisfying $M\subset M'$ and $p^{c}M^{\vee}=M$. The map
$\epsilon$ induces the natural injection of this set into $\mathcal{N}_{red}'(\overline{\mathbb{F}})$.
If $f:\underline{A}\rightarrow\underline{A}'$ lies in $t_{c}^{-1}(x)$,
then $\rho_{X}^{*}\underline{A}'=\underline{A}$. Embed $\mathcal{N}_{red}'$
into $\mathcal{N}_{red}'\times G(\mathbb{A}_{f})$ by $\alpha:z\mapsto(z,1)$.
There is a commutative diagram:

$$\xymatrix{I(\mathbb{Q})\diagdown\mathcal{N}{}_{red}'\times G(\mathbb{A}_{f}^{p})\diagup K^{p}\ar[r]^-{\Theta} & \overline{Sh_{K}}^{ss}\times k\\ \mathcal{N}{}_{red}'\ar[u]^{\alpha} & t_{c}^{-1}(x)\ar[l]_{\epsilon}\ar[u]_{s}}$$

\begin{prop}
The restriction of $s$ to $t_{c}^{-1}(x)$ is a finite morphism.\end{prop}
\begin{proof}
The restriction of $\alpha$ to any quasi-compact subscheme of $\mathcal{N}_{red}'$
is finite (see \cite{V2}, 5.4).\end{proof}
\begin{cor}
The morphism 
\[
p\textrm{-}\mathscr{I}sog^{(c),ss}\times\kappa(\mathcal{O}_{E_{p}})\overset{(s,t)}{\longrightarrow}\overline{Sh_{K}}^{ss}\times\overline{Sh_{K}}^{ss}
\]
 is finite.\end{cor}
\begin{proof}
It is proper and quasi-finite.\end{proof}
\begin{cor}
Let $x\in\overline{Sh_{K}}^{ss}(\overline{\mathbb{F}})$ and $c\geq2$
be an even integer. The dimension of the fibre $t_{c}^{-1}(x)$ is
$\frac{n-1}{2}$.\end{cor}
\begin{proof}
Clearly $\textrm{dim}(t_{c}^{-1}(x))\leq\frac{n-1}{2}$. Let $M$
be the Dieudonné module associated to $x$. Let $M_{0}$ be any quasi-unitary
Dieudonné module of signature $(n,0)$ such that $M_{0}\subset M$
and $M_{0}^{\vee}=p^{c-1}M_{0}$. The irreducible component of $\mathcal{N}_{red}'$
associated to $M_{0}$ is then contained in $t_{c}^{-1}(x)$.\end{proof}
\begin{rem}
\label{t2}When $c=2$ and $x$ has type $n$, there is only one such
$M_{0}$ (namely $p\Lambda^{+}(M)$). Therefore, the fibre $t_{c}^{-1}(x)$
has only one irreducible component of dimension $\frac{n-1}{2}$.
\end{rem}

\subsection{Irreducible components of $p\textrm{-}\mathscr{I}sog\times\overline{\mathbb{F}}$}
\begin{prop}
\label{comppisog}$p\textrm{-}\mathscr{I}sog\times\overline{\mathbb{F}}$
is equidimensional of dimension $n-1$. If an irreducible component
of $p\textrm{-}\mathscr{I}sog\times\overline{\mathbb{F}}$ intersects
the ordinary locus, it is contained in the closure of $p\textrm{-}\mathscr{I}sog^{ord}\times\overline{\mathbb{F}}$.
Otherwise, all its points are supersingular.\end{prop}
\begin{proof}
Let $C$ be an irreducible component of $p\textrm{-}\mathscr{I}sog^{(c)}\times\overline{\mathbb{F}}$
for $c\geq0$. Using \cite{BW1} (proposition 6.15), we have $\textrm{dim}(C)\geq n-1$.
Suppose that $C$ intersects $p\textrm{-}\mathscr{I}sog^{(c),ord}\times\overline{\mathbb{F}}$.
Since the ordinary locus is open, $C$ is contained in its closure
and $\textrm{dim}(C)=n-1$. Suppose that $C$ has no ordinary point,
and that there exists a non-supersingular point $z\in C$. There is
an open subset $U\subset\overline{Sh_{K}}$ containing $t(z)$ such
that $U\cap\overline{Sh_{K}}^{ss}=\emptyset$. Then $z$ is in $t^{-1}(U)\cap C$
so this is a non empty dense open subset of $C$. By \cite{BW1} corollary
7.3, the map $t$ is finite over $t^{-1}(U)$, thus:
\[
\textrm{dim}(t^{-1}(U)\cap C)=\textrm{dim}(t(t^{-1}(U)\cap C))\leq\textrm{dim}(t(C))<n-1
\]
because $t(C)$ doesn't meet $\overline{Sh_{K}}^{ord}$. This contradicts
$\textrm{dim}(C)\geq n-1$. We have shown that a component not intersecting
$p\textrm{-}\mathscr{I}sog^{(c),ord}\times\overline{\mathbb{F}}$
is supersingular. Finally, let $C$ be a supersingular irreducible
component. We have $\textrm{dim}(\overline{Sh_{K}}^{ss})=\frac{n-1}{2}=\textrm{dim}(t_{c}^{-1}(x))$
for all $x\in\overline{Sh_{K}}^{ss}(\overline{\mathbb{F}})$, so $\textrm{dim}(C)=n-1$.
\end{proof}

\subsection{The map $(s,t)$}
\begin{thm}
\label{stimfer}Let $c\geq0$ be an even integer. There exists $K^{p}\subset G(\mathbb{A}_{f}^{p})$
such that
\[
p\textrm{-}\mathscr{I}sog_{K}^{(c),ss}\times\kappa(\mathcal{O}_{E_{p}})\overset{(s,t)}{\longrightarrow}\overline{Sh_{K}}^{ss}\times\overline{Sh_{K}}^{ss}
\]
is a closed immersion.\end{thm}
\begin{proof}
We will use the moduli problems $\mathfrak{F},\mathfrak{I}$ described
in 2.1 and 3.1. Choose an $\mathcal{O}_{E}\left[\frac{1}{p}\right]$-lattice
$L\subset V$ satisfying the condition (\ref{condL}). Let $x_{i}=(A_{i},\lambda_{i},\iota_{i},\overline{\eta_{i}})$,
$i\in\left\{ 0,1\right\} $ be two points of $\mathfrak{F}(\overline{\mathbb{F}})$
with multiplicator $c$. We assume that there is an isogeny $h:(A_{0},\lambda_{0},\iota_{0})\rightarrow(A_{1},\lambda_{1},\iota_{1})$.
We write $R=\textrm{Hom}(A_{0},A_{1})$ for the group of homomorphisms
(with no compatibility condition). If $f,g\in R$, we write 
\[
\left\langle f,g\right\rangle =Tr\left(d^{2}\lambda_{0}^{-1}\circ f^{\vee}\circ\lambda_{1}\circ g\right)
\]
where $\textrm{Tr}:\textrm{End}(A_{0})\rightarrow\mathbb{Z}$ is the
trace morphism. Since $\lambda_{0}$ has degree $d^{2}$, the quasi-isogeny
$d^{2}\lambda_{0}^{-1}$ is an isogeny, thus $\left\langle f,g\right\rangle \in\mathbb{Z}$.
The form $\left\langle ,\right\rangle $ is symmetric, positive definite.
Indeed, $h$ identifies this form with one on $\textrm{End}(A_{0})\otimes\mathbb{Q}$
which is positive definite because the Rosati involution is. Define
$q(f)=\left\langle f,f\right\rangle $ for $f\in R$. For a $p$-isogeny
$f$ of multiplicator $c$, we have 
\begin{equation}
q(f)=d^{2}p^{c}.\label{qduG}
\end{equation}
Let $\ell\neq p$ be a prime. We have an injective morphism: 
\[
R\otimes\mathbb{Z}_{\ell}\hookrightarrow\textrm{Hom}_{\mathbb{Z}_{\ell}}\left(T_{\ell}(A_{0}),T_{\ell}(A_{1})\right).
\]
Each module $T_{\ell}(A_{i})$ carries a pairing $\varepsilon^{\lambda_{i}}:T_{\ell}(A_{i})\times T_{\ell}(A_{i})\rightarrow\mathbb{Z}_{l}(1)$
associated to $\lambda_{i}$. If $f:V_{\ell}(A_{0})\rightarrow V_{\ell}(A_{1})$
is a $\mathbb{Q}_{\ell}$-linear map, we define $f^{*}:V_{\ell}(A_{1})\rightarrow V_{\ell}(A_{0})$
by the formula:
\[
\left\langle fx,y\right\rangle {}^{\lambda_{1}}=\left\langle x,f^{*}y\right\rangle {}^{\lambda_{0}}
\]
and we write $\widehat{q}(f)=d^{2}\textrm{Tr}(f^{*}f)$. If $f\in R$,
the properties of the Weil pairing show that $f^{*}=\lambda_{0}^{-1}\circ f^{\vee}\circ\lambda_{1}$.
This implies $q(f)=\widehat{q}(f)$. The level structures $\overline{\eta_{0}}$,
$\overline{\eta_{1}}$ identify $\textrm{Hom}_{\mathbb{Z}_{\ell}}\left(T_{\ell}(A_{0}),T_{\ell}(A_{1})\right)$
with $\textrm{End}_{\mathbb{Z}_{\ell}}\left(L\otimes\mathbb{Z}_{\ell}\right)$.
Again, we define an involution $*$ of $\textrm{End}_{\mathbb{Q}}\left(V\right)$
by the formula:
\[
\psi(fx,y)=\psi(x,f^{*}y),\quad\forall f\in\textrm{End}_{\mathbb{Q}}\left(V\right),\forall x,y\in V
\]
and we write $q_{0}(f)=d^{2}\textrm{Tr}(ff^{*})$. We get a commutative
diagram

\begin{displaymath}
\xymatrix{
R\otimes\mathbb{Z}_{\ell} \ar[dr]_{q \otimes\mathbb{Z}_{\ell}} \ar@{^{(}->}[r]  & \textrm{Hom}_{\mathbb{Z}_{\ell}}\left(T_{\ell}(A_{0}),T_{\ell}(A_{1})\right)\ar[d]^{\widehat{q}} \ar[r]^(.57){\simeq} & \textrm{End}_{\mathbb{Z}_{\ell}}\left(L\otimes\mathbb{Z}_{\ell}\right)\ar[dl]^{q_0\otimes\mathbb{Z}_{\ell} } \\
  & \mathbb{Z}_{\ell} & }
\end{displaymath}From now on, we assume that $A_{0}$, $A_{1}$ are supersingular.
Then we have an isomorphism

\[
\textrm{Hom}(A_{0},A_{1})\otimes\mathbb{Z}_{\ell}\overset{\sim}{\rightarrow}\textrm{Hom}_{\mathbb{Z}_{\ell}}\left(T_{\ell}(A_{0}),T_{\ell}(A_{1})\right).
\]
The diagram above shows $(R,q)\otimes\mathbb{Z}_{\ell}\simeq(\textrm{End}_{\mathbb{Z}}(L),q_{0})\otimes\mathbb{Z}_{\ell}$
as $\mathbb{Z}_{\ell}$-quadratic spaces. At $p$, we have an isomorphism:

\begin{equation}
\textrm{Hom}(A_{0},A_{1})\otimes\mathbb{Z}_{p}\simeq\textrm{Hom}_{W[F,V]}\left(\mathbb{D}(A_{0}),\mathbb{D}(A_{1})\right)\label{isomp}
\end{equation}
because $A_{0}$, $A_{1}$ are supersingular. The pairings $\left\langle ,\right\rangle _{i}:\mathbb{D}(A_{i})\times\mathbb{D}(A_{i})\rightarrow W$
induce a transformation $f\mapsto f*$ : 
\[
\textrm{Hom}\left(\mathbb{D}(A_{0})\otimes W_{\mathbb{Q}},\mathbb{D}(A_{1})\otimes W_{\mathbb{Q}}\right)\rightarrow\textrm{Hom}\left(\mathbb{D}(A_{1})\otimes W_{\mathbb{Q}},\mathbb{D}(A_{0})\otimes W_{\mathbb{Q}}\right)
\]
Then $q\otimes\mathbb{Z}_{p}$ becomes via (\ref{isomp}) the form
$f\mapsto d^{2}\textrm{Tr}\left(ff^{*}\right)$ on $\textrm{Hom}_{W[F,V]}\left(\mathbb{D}(A_{0}),\mathbb{D}(A_{1})\right)$.
Since the pairings are perfect, the discriminant of $q\otimes\mathbb{Z}_{p}$
is a unit in $\mathbb{Z}_{p}$.

To every $x_{0},x_{1}\in\mathfrak{F}(\overline{\mathbb{F}})$ (such
that there exists an isogeny $h:x_{0}\rightarrow x_{1}$ in $\mathfrak{I}(\overline{\mathbb{F}})$),
we have associated a positive definite $\mathbb{Z}$-quadratic space
$(R,q)$, whose discriminant is independant of $x_{0},x_{1},h$. Let
$\mathcal{L}$ be the set of isomorphism classes of the spaces $(R,q)$.
It is finite since there are only finitely many isomorphism classes
of positive definite $\mathbb{Z}$-quadratic spaces of given discriminant.
For every fixed $C\geq0$, we thus can find $N\geq0$ coprime to $p$
such that:
\[
\forall(R,q)\in\mathcal{L},\:\forall u,v\in R,\:\begin{cases}
u\equiv v\,[NR]\\
q(u)=q(v)=C
\end{cases}\Longrightarrow u=v.
\]
Take $N$ satisfying this condition for $C=d^{2}p^{c}$. Define $K'^{p}=\left\{ g\in K^{p},\:(g-1)(\widehat{L}^{(p)})\subset N\widehat{L}^{(p)}\right\} $
and write $\mathfrak{I}_{K'^{p}}$ the moduli problem for this new
level structure. Let $f,g$ be two isogenies in $\mathfrak{I}_{K'^{p}}(\overline{\mathbb{F}})$
of multiplicator $c$, such that $s(f)=s(g)$ and $t(f)=t(g)$, denoted
respectively $x_{0}$ and $x_{1}$. Then (\ref{qduG}) shows that
$q(f)=q(g)=C$. By definition of $K'^{p}$, we have $f=g$ on $A_{0}[N]$.
There exists $h\in R$ such that $f-g=Nh$, thus $f=g$. This shows
that $(s,t)$ is injective on supersingular $\overline{\mathbb{F}}$-points.

Write $R=\frac{\overline{\mathbb{F}}\left[t\right]}{(t^{2})}$. Let
$f,g$ be two supersingular points in $\mathfrak{I}_{K'^{p}}(R)$
such that $s(f)=s(g)$ and $t(f)=t(g)$. We denote by $f_{0}$, $g_{0}$
the reduced isogenies on $\overline{\mathbb{F}}$. We have $(s,t)(f_{0})=(s,t)(g_{0})$,
thus $f_{0}=g_{0}$, and we deduce $f=g$ by \cite{GZR}, lemma 3.1.
This shows that $(s,t)$ is a closed immersion.
\end{proof}

\section{Congruence relation}

\subsection{A few lemmas}
\begin{thm}
\label{thmfibre}Let $X,Y$ be irreducible schemes of finite type
over a field. Let $f:X\rightarrow Y$ be a dominant morphism. Then
there is an open dense subset $U\subset Y$ such that for all $y\in U$,
we have:\textup{
\[
\textrm{dim}(f^{-1}(y))=\textrm{dim}(X)-\textrm{dim}(Y).
\]
}\end{thm}
\begin{proof}
We may assume that $X,Y$ are reduced. Using \cite{EGA4} théorème
6.9.1, there exists an open dense subset $U\subset Y$ such that $f:f^{-1}(U)\rightarrow U$
is flat. Then use lemme 13.1.1 and corollaire 14.2.4 of \textsl{ibid}.\end{proof}
\begin{cor}
\label{dimX}Let $X,Y$ be schemes of finite type over a field. Let
$f:X\rightarrow Y$ be a dominant morphism and $r\geq0$. Assume that
for all $y\in Y$, the dimension of $f^{-1}(y)$ is $r$. Then \textup{$\textrm{dim}(X)-\textrm{dim}(Y)=r$}.\end{cor}
\begin{proof}
This is a simple exercise.\end{proof}
\begin{lem}
Let $C\subset p\textrm{-}\mathscr{I}sog^{(c)}\times\overline{\mathbb{F}}$
be a supersingular irreducible component. Then $C_{s}:=s(C)$ and
$C_{t}:=t(C)$ are irreducible components of $\overline{Sh_{K}}^{ss}\times\overline{\mathbb{F}}$.\end{lem}
\begin{proof}
They are irreducible closed subsets of dimension $\geq\frac{n-1}{2}$
since the fibres have dimension $\leq\frac{n-1}{2}$. But $\overline{Sh_{K}}^{ss}$
is equidimensional of dimension $\frac{n-1}{2}$ (\cite{V2} theorem
5.2), so the result follows.\end{proof}
\begin{prop}
\label{compunik}Let $C_{1},C_{2}\subset p\textrm{-}\mathscr{I}sog^{(c)}\times\overline{\mathbb{F}}$
be supersingular irreducible components. Assume that the map $(s,t)$
is a closed immersion on $p\textrm{-}\mathscr{I}sog_{K}^{(c),ss}$.
Assume further that $C_{1,s}=C_{2,s}$ and $C_{1,t}=C_{2,t}$. Then
$C_{1}=C_{2}$.\end{prop}
\begin{proof}
The map $(s,t)$ induces a closed immersion $C_{1}\hookrightarrow C_{1,s}\times C_{1,t}$.
Since $C_{1,s}$ and $C_{1,t}$ are irreducible components of $\overline{Sh_{K}}^{ss}\times\overline{\mathbb{F}}$,
they are smooth of dimension $\frac{n-1}{2}$. The product $C_{1,s}\times C_{1,t}$
is thus irreducible of dimension $n-1$, so $(s,t)$ defines an isomorphism
$C_{1}\overset{\sim}{\rightarrow}C_{1,s}\times C_{1,t}$. The same
holds for $C_{2}$, and we deduce the result.
\end{proof}

\subsection{The Frobenius action}

Let $\mathcal{F}$ be the Frobenius map on $\overline{Sh_{K}}$ and
$p\textrm{-}\mathscr{I}sog\times\kappa(\mathcal{O}_{E_{p}})$. If
$C$ is a cycle, write $\left|C\right|$ for its support.
\begin{prop}
\label{pCFC}Let $\widetilde{C}$ be an irreducible component of $\overline{Sh_{K}}^{ss}\times\overline{\mathbb{F}}$.
We have:
\[
\mathcal{F}(\widetilde{C})=\left\langle p\right\rangle \widetilde{C}.
\]
\end{prop}
\begin{proof}
Let $x\in\overline{Sh_{K}}^{ss}(\overline{\mathbb{F}})$ be a point
of type $n$ whose image lies in $\widetilde{C}$. Write $M=\mathbb{D}(x)$.
Then $t_{2}^{-1}(x)$ has a unique irreducible component $C$ of dimension
$\frac{n-1}{2}$ (see remark \ref{t2}). More precisely, points $y\in p\textrm{-}\mathscr{I}sog^{(2)}(\overline{\mathbb{F}})$
whose image lie in $C$ correspond to quasi-unitary Dieudonné modules
$M'$ of signature $(n-1,1)$ satisfying $p^{2}M'^{\vee}=M'$ and
$M'\subset p\Lambda^{+}(M)$. Clearly the isogenies $p:\left\langle p^{-1}\right\rangle x\rightarrow x$
and $V:\left\langle p^{-2}\right\rangle Fx\rightarrow x$ belong to
$t_{2}^{-1}(x)$. They lie in $C$ because $V^{2}M\subset p\Lambda^{+}(M)$
and $pM\subset p\Lambda^{+}(M)$. Thus $\left\langle p^{-2}\right\rangle \mathcal{F}x$
and $\left\langle p^{-1}\right\rangle x$ lie in $s(C)$, which is
an irreducible component of $\overline{Sh_{K}}^{ss}\times\overline{\mathbb{F}}$.
Observe that $\left\langle p^{-2}\right\rangle \mathcal{F}x\in\left\langle p^{-2}\right\rangle \mathcal{F}(\widetilde{C})$
and $\left\langle p^{-1}\right\rangle x\in\left\langle p^{-1}\right\rangle \widetilde{C}$.
We deduce $\left\langle p^{-2}\right\rangle \mathcal{F}(\widetilde{C})=s(C)=\left\langle p^{-1}\right\rangle \widetilde{C}$
because a point of type $n$ lies in a unique irreducible component
of $\overline{Sh_{K}}^{ss}\times\overline{\mathbb{F}}$.\end{proof}
\begin{prop}
\label{FCCF}Let $C\subset p\textrm{-}\mathscr{I}sog^{(c)}\times\overline{\mathbb{F}}$
be an irreducible component. Then
\[
F\cdot\mathcal{F}(C)=C\cdot F
\]
\end{prop}
\begin{proof}
If $\underline{A}_{1}\overset{f}{\rightarrow}\underline{A}_{0}$ is
an $\overline{\mathbb{F}}$-valued point of $p\textrm{-}\mathscr{I}sog^{(c)}\times\overline{\mathbb{F}}$
whose image lies in $C$, then $\mathcal{F}(f)$ is the isogeny $\underline{A}_{1}^{(q)}\overset{f^{(q)}}{\rightarrow}\underline{A}_{0}^{(q)}$
and we have: $f^{(q)}\circ F_{A_{1}}=F_{A_{0}}\circ f$ , where $F_{A_{i}}:A_{i}\rightarrow A_{i}^{(q)}$
is the Frobenius isogeny. This shows $\left|F\cdot\mathcal{F}(C)\right|=\left|C\cdot F\right|$.
Write $X$ for this support. We define an isomorphism 
\[
\alpha:p\textrm{-}\mathscr{I}sog^{(c)}\times\overline{\mathbb{F}}\longrightarrow\left(p\textrm{-}\mathscr{I}sog^{(c)}\times\overline{\mathbb{F}}\right)\times_{t,s}F
\]
by sending $\underline{A}_{1}\overset{f}{\longrightarrow}\underline{A}_{0}$
to the pair $\left(\underline{A}_{1}\overset{f}{\longrightarrow}\underline{A}_{0}\:,\:\underline{A}_{0}\overset{F}{\longrightarrow}\underline{A}_{0}^{(q)}\right)$.
Similarly, define 
\[
\beta:p\textrm{-}\mathscr{I}sog^{(c)}\times\overline{\mathbb{F}}\longrightarrow F\times_{t,s}\left(p\textrm{-}\mathscr{I}sog^{(c)}\times\overline{\mathbb{F}}\right)
\]
by sending $\underline{A}_{1}\overset{f}{\longrightarrow}\underline{A}_{0}$
to the pair $\left(\underline{A}_{1}\overset{F}{\longrightarrow}\underline{A}_{1}^{(q)}\:,\:\underline{A}_{1}^{(q)}\overset{f^{(q)}}{\longrightarrow}\underline{A}_{0}^{(q)}\right)$.
It has degree 1. Consider the commutative diagram

$$\xymatrix{\relax
C \ar[d]_{\beta} \ar[r]^{\alpha} & C\times_{t,s}F \ar[d]_-{c_2} \\
F\times_{t,s}\mathcal{F}(C) \ar[r]^-{c_1} & X
}$$where $c_{1}$ and $c_{2}$ are the restriction of $c$ to $F\times_{t,s}\mathcal{F}(C)$
and $C\times_{t,s}F$ respectively. We have $F\cdot\mathcal{F}(C)=\textrm{deg}(c_{1})X$
and $C\cdot F=\textrm{deg}(c_{2})X$. Since $\alpha$ and $\beta$
have degree 1, we deduce $F\cdot\mathcal{F}(C)=C\cdot F$.
\end{proof}
We have proved some results on $p\textrm{-}\mathscr{I}sog^{(c)}\times\overline{\mathbb{F}}$.
Observe that diagram (\ref{diagcommutatif}) involves $p\textrm{-}\mathscr{I}sog^{(c)}\times\kappa(\mathcal{O}_{E_{p}})$.
For the relation $H_{p}(F)=0$ to make sense, $F$ has to commute
with the coefficients of $H_{p}$. The pull-back by $p\textrm{-}\mathscr{I}sog^{(c)}\times\overline{\mathbb{F}}\rightarrow p\textrm{-}\mathscr{I}sog^{(c)}\times\kappa(\mathcal{O}_{E_{p}})$
defines a $\mathbb{Q}$-algebra homomorphism 
\begin{equation}
\mathbb{Q}\left[p\textrm{-}\mathscr{I}sog^{(c)}\times\kappa(\mathcal{O}_{E_{p}})\right]\hookrightarrow\mathbb{Q}\left[p\textrm{-}\mathscr{I}sog^{(c)}\times\overline{\mathbb{F}}\right].\label{injkf}
\end{equation}

\begin{cor}
The element $F$ belongs to the centre of $\mathbb{Q}\left[p\textrm{-}\mathscr{I}sog^{(c)}\times\kappa(\mathcal{O}_{E_{p}})\right]$.\end{cor}
\begin{proof}
This follows from proposition \ref{FCCF} using (\ref{injkf}).
\end{proof}

\subsection{Etale covering}

Let $K^{p}$ and $K'^{p}$ be two compact open subgroups of $G(\mathbb{A}_{f}^{p})$,
such that $K'^{p}\subset K^{p}$. Write $K=K_{p}K^{p}$ and $K'=K_{p}K'^{p}$.
Then we have an étale coverings

\begin{eqnarray*}
\pi:Sh_{K'} & \longrightarrow & Sh_{K}\\
\Pi:p\textrm{-}\mathscr{I}sog_{K'} & \longrightarrow & p\textrm{-}\mathscr{I}sog_{K}
\end{eqnarray*}

\begin{lem}
\label{piF}The push-forward by $\Pi$ defines a $\mathbb{Q}$-algebra
homomorphism:
\[
\Pi_{*}:\mathbb{Q}\left[p\textrm{-}\mathscr{I}sog_{K'}^{(c)}\times k\right]\longrightarrow\mathbb{Q}\left[p\textrm{-}\mathscr{I}sog_{K}^{(c)}\times k\right]
\]
Further, $\Pi_{*}(F)=\textrm{deg}(\pi)F$ and $\Pi_{*}(\left\langle p\right\rangle )=\textrm{deg}(\pi)\left\langle p\right\rangle $.\end{lem}
\begin{proof}
Consider the commutative diagram:

$$\xymatrix{p\textrm{-}\mathscr{I}sog_{K'}^{(c)}\times p\textrm{-}\mathscr{I}sog_{K'}^{(c)}\ar[r]^-{c}\ar[d]^{\Pi\times\Pi} & p\textrm{-}\mathscr{I}sog_{K'}^{(c)}\ar[d]^{\Pi}\\p\textrm{-}\mathscr{I}sog_{K}^{(c)}\times_{s,t}p\textrm{-}\mathscr{I}sog_{K}^{(c)}\ar[r]^-{c} & p\textrm{-}\mathscr{I}sog_{K}^{(c)}}$$If
$C_{1},C_{2}$ are cycles,
\[
\Pi_{*}(C_{1}\cdot C_{2})=\Pi_{*}c_{*}(C_{1}\times_{t,s}C_{2})=c_{*}(\Pi\times\Pi)_{*}(C_{1}\times_{t,s}C_{2})=c_{*}(\Pi_{*}(C_{1})\times_{t,s}\Pi_{*}(C_{2}))=\Pi_{*}(C_{1})\cdot\Pi_{*}(C_{2})
\]
thus $\Pi_{*}$ is a ring homomorphism. We have another commutative
diagram:

$$\xymatrix{p\textrm{-}\mathscr{I}sog_{K'}^{(2)}\times k \ar[r]^{\Pi} & p\textrm{-}\mathscr{I}sog_{K}^{(2)}\times k\\
\overline{Sh_{K'}}\times k \ar[u]^F \ar[r]^{\pi} & \overline{Sh_{K}}\times k \ar[u]^F
}$$thus $\Pi_{*}(F)=\textrm{deg}(\pi)F$ and similarly $\Pi_{*}(\left\langle p\right\rangle )=\textrm{deg}(\pi)\left\langle p\right\rangle $.
\end{proof}

\subsection{Main theorem}
\begin{lem}
\label{RDClemme}Let $C\subset p\textrm{-}\mathscr{I}sog^{(c)}\times\overline{\mathbb{F}}$
be a supersingular irreducible component. Assume that the map $(s,t)$
is a closed immersion on $p\textrm{-}\mathscr{I}sog_{K}^{(c),ss}$.
In the ring $\mathbb{Q}\left[p\textrm{-}\mathscr{I}sog\times\overline{\mathbb{F}}\right]$,
the following relation holds:
\[
C\cdot(F-p^{n-1}\left\langle p\right\rangle )=0
\]
\end{lem}
\begin{proof}
The proof is twofold: First we show that $C\cdot F$ and $C\cdot\left\langle p\right\rangle $
have the same support, then we look at multiplicities. The supports
$|C\cdot F|$ and $|C\cdot\left\langle p\right\rangle |$ are irreducible,
of dimension $n-1$. Indeed, they are the direct images by the composition
morphism $c$ of $C\times_{t,s}F$ and $C\times_{t,s}\left\langle p\right\rangle $
respectively, which are irreducible. Thus, $|C\cdot F|$ and $|C\cdot\left\langle p\right\rangle |$
are irreducible components of $p\textrm{-}\mathscr{I}sog^{(c+2)}\times\overline{\mathbb{F}}$.
We clearly have $s(C\cdot F)=s(C\cdot\left\langle p\right\rangle )$.
Using proposition \ref{pCFC}, we have:
\[
t(C\cdot F)=\mathcal{F}(C_{t})=\left\langle p\right\rangle C_{t}=t(C\cdot\left\langle p\right\rangle )
\]
Proposition \ref{compunik} then shows that $|C\cdot F|=|C\cdot\left\langle p\right\rangle |$.
We denote by $X$ this closed subset.

The projection on $C$ defines isomorphisms $a_{F}:C\times_{t,s}F\rightarrow C$
and $a_{p}:C\times_{t,s}\left\langle p\right\rangle \rightarrow C$.
Write $c_{F}=c\circ a_{F}^{-1}$ and $c_{p}=c\circ a_{p}^{-1}$. There
is a commutative diagram:

$$\xymatrix{\relax C\times_{t,s}F\ar[r]^-{c} & X \ar[r]^-{(s,t)} & C_s\times \mathcal{F}(C_t) \\C\ar[u]^{\simeq}\ar[d]_{\simeq}\ar[rd]_{c_{p}}\ar[ru]^{c_{F}} \ar[r]^{(s,t)} & C_s \times C_t \ar[rd]^{\textrm{id}\times \left\langle p\right\rangle } \ar[ru]_{\textrm{id}\times\mathcal{F}} &\\C\times_{t,s}\left\langle p\right\rangle \ar[r]_-{c} & X \ar[r]_-{(s,t)} & C_s\times \mathcal{F}(C_t)}$$Recall
that $\mathcal{F}(C_{t})=\left\langle p\right\rangle C_{t}$. By definition,
$C\cdot F=\textrm{deg}(c_{F})X$ and $C\cdot\left\langle p\right\rangle =\textrm{deg}(c_{p})X$.
The diagram shows that $\frac{\textrm{deg}(c_{F})}{\textrm{deg}(c_{p})}=\frac{\textrm{deg}(\textrm{id}\times\mathcal{F})}{\textrm{deg}(\textrm{id}\times\left\langle p\right\rangle )}$.
The map $\left\langle p\right\rangle :C_{t}\rightarrow\left\langle p\right\rangle C_{t}$
has degree 1 and $\mathcal{F}:C_{t}\rightarrow\mathcal{F}(C_{t})$
has degree $p^{2\frac{n-1}{2}}=p^{n-1}$ since $C_{t}$ has dimension
$\frac{n-1}{2}$. Thus, $\textrm{deg}(c_{F})=p^{n-1}\textrm{deg}(c_{p})$,
and finally $C\cdot F=p^{n-1}C\cdot\left\langle p\right\rangle $.\end{proof}
\begin{thm}
\label{rdcsjalv}Let $C\subset p\textrm{-}\mathscr{I}sog^{(c)}\times\overline{\mathbb{F}}$
be an irreducible supersingular component. In the ring $\mathbb{Q}\left[p\textrm{-}\mathscr{I}sog^{(c)}\times\overline{\mathbb{F}}\right]$,
the following relation holds:
\[
C\cdot(F-p^{n-1}\left\langle p\right\rangle )=0.
\]
\end{thm}
\begin{proof}
Let $K'^{p}\subset K^{p}$ such that $(s,t)$ is a closed immersion
on $p\textrm{-}\mathscr{I}sog_{K'}^{(c+2)}\times k$ (proposition
\ref{stimfer}), and $C'$ be a supersingular irreducible component
of $p\textrm{-}\mathscr{I}sog^{(c)}\times k$ such that $\Pi(C')=C$.
We have $C'\cdot(F-p^{n-1}\left\langle p\right\rangle )=0$ (lemma
\ref{RDClemme}), and taking the image by $\Pi_{*}$, we find $C\cdot(F-p^{n-1}\left\langle p\right\rangle )=0$
(lemma \ref{piF}).\end{proof}
\begin{thm}
\label{rdcfinal}Let $H_{p}$ be the Hecke polynomial. Consider the
coefficients of $H_{p}$ in $\mathbb{Q}\left[p\textrm{-}\mathscr{I}sog\times\kappa(\mathcal{O}_{E_{p}})\right]$
through $\sigma\circ h$ (see diagram (\ref{diagcommutatif})). We
have the relation:
\[
H_{p}(F)=0.
\]
\end{thm}
\begin{proof}
We have $H_{p}(t)=R(t)\cdot(t-p^{n-1}\left\langle p\right\rangle )$
with $R(t)\in\mathbb{Q}\left[p\textrm{-}\mathscr{I}sog\times\kappa(\mathcal{O}_{E_{p}})\right][t]$
(theorem \ref{facteurH}). Let $H_{p}^{'}=\textrm{cl}(\textrm{ord}(H_{p}))$
and $R'=\textrm{cl}(\textrm{ord}(R))$. Then $H_{p}^{'}(t)=R^{'}(t)\cdot(t-p^{n-1}\left\langle p\right\rangle )$.
Theorem \ref{moonenRDCord} shows that $H'_{p}(F)=0$ in $\mathbb{Q}\left[p\textrm{-}\mathscr{I}sog\times\kappa(\mathcal{O}_{E_{p}})\right]$.
Therefore,
\[
H_{p}(F)=\left(H_{p}-H_{p}^{'}\right)(F)=\left(R-R^{'}\right)(F)\cdot\left(F-p^{n-1}\left\langle p\right\rangle \right)
\]
The coefficients of $H_{p}$ and $R$ are linear combinations of supersingular
irreducible components of $p\textrm{-}\mathscr{I}sog\times\kappa(\mathcal{O}_{E_{p}})$.
Indeed, they are specialization of cycles of dimension $n-1$ in $\mathbb{Q}\left[p\textrm{-}\mathscr{I}sog\times E_{p}\right]$,
and specialization respects dimensions (\cite{fulton}, 20.3). These
componants are either ordinary or supersingular (proposition \ref{comppisog}).
Thus, the coefficients of $R-R^{'}$ are $n-1$-dimensional supersingular
cycles, and so is $\left(R-R^{'}\right)(F)$. Finally, theorem \ref{rdcsjalv}
shows that $H_{p}(F)=0$.
\end{proof}
\bibliographystyle{plainnat}
\bibliography{references}

\begin{thebibliography}{15}
\providecommand{\natexlab}[1]{#1}
\providecommand{\url}[1]{\texttt{#1}}
\expandafter\ifx\csname urlstyle\endcsname\relax
  \providecommand{\doi}[1]{doi: #1}\else
  \providecommand{\doi}{doi: \begingroup \urlstyle{rm}\Url}\fi

\bibitem[Blasius and Rogawski(1994)]{BR}
Don Blasius and Jonathan~D. Rogawski.
\newblock {Zeta functions of {S}himura varieties}.
\newblock In \emph{{Motives ({S}eattle, {WA}, 1991)}}, volume~55 of
  \emph{{Proc. Sympos. Pure Math.}}, pages 525--571. Amer. Math. Soc.,
  Providence, RI, 1994.

\bibitem[B{\"u}ltel(2002)]{bult}
Oliver B{\"u}ltel.
\newblock {The congruence relation in the non-{PEL} case}.
\newblock \emph{J. Reine Angew. Math.}, 544:\penalty0 133--159, 2002.
\newblock ISSN 0075-4102.
\newblock \doi{10.1515/crll.2002.020}.
\newblock URL \url{http://dx.doi.org/10.1515/crll.2002.020}.

\bibitem[B{\"u}ltel and Wedhorn(2006)]{BW1}
Oliver B{\"u}ltel and Torsten Wedhorn.
\newblock {Congruence relations for {S}himura varieties associated to some
  unitary groups}.
\newblock \emph{J. Inst. Math. Jussieu}, 5\penalty0 (2):\penalty0 229--261,
  2006.
\newblock ISSN 1474-7480.
\newblock \doi{10.1017/S1474748005000253}.
\newblock URL \url{http://dx.doi.org/10.1017/S1474748005000253}.

\bibitem[Conrad(2004)]{GZR}
Brian Conrad.
\newblock {Gross-{Z}agier revisited}.
\newblock In \emph{{Heegner points and {R}ankin {$L$}-series}}, volume~49 of
  \emph{{Math. Sci. Res. Inst. Publ.}}, pages 67--163. Cambridge Univ. Press,
  Cambridge, 2004.
\newblock \doi{10.1017/CBO9780511756375.006}.
\newblock URL \url{http://dx.doi.org/10.1017/CBO9780511756375.006}.
\newblock With an appendix by W. R. Mann.

\bibitem[Faltings and Chai(1990)]{fc}
G.~Faltings and C.-L. Chai.
\newblock \emph{{Degeneration of abelian varieties}}, volume~22.
\newblock Springer-Verlag, 1990.

\bibitem[Fulton(1998)]{fulton}
W.~Fulton.
\newblock \emph{{Intersection Theory}}.
\newblock {Ergebnisse Der Mathematik Und Ihrer Grenzgebiete, 3. Folge, Bd. 2}.
  Springer-Verlag GmbH, 1998.
\newblock ISBN 9783540620464.

\bibitem[Grothendieck(1966)]{EGA4}
A.~Grothendieck.
\newblock {{\'E}l{\'e}ments de g{\'e}om{\'e}trie alg{\'e}brique. {IV}.
  {\'E}tude locale des sch{\'e}mas et des morphismes de sch{\'e}mas. {III}}.
\newblock \emph{Inst. Hautes {\'E}tudes Sci. Publ. Math.}, \penalty0
  (28):\penalty0 255, 1966.
\newblock ISSN 0073-8301.

\bibitem[Grothendieck(1974)]{grothBT}
Alexandre Grothendieck.
\newblock \emph{{Groupes de {B}arsotti-{T}ate et cristaux de {D}ieudonn{\'e}}}.
\newblock Les Presses de l'Universit{\'e} de Montr{\'e}al, Montreal, Que.,
  1974.
\newblock S{\'e}minaire de Math{\'e}matiques Sup{\'e}rieures, No. 45
  ({\'E}t{\'e}, 1970).

\bibitem[Kottwitz(1992)]{Kot}
Robert~E. Kottwitz.
\newblock {Points on some {S}himura varieties over finite fields}.
\newblock \emph{J. Amer. Math. Soc.}, 5\penalty0 (2):\penalty0 373--444, 1992.
\newblock ISSN 0894-0347.
\newblock \doi{10.2307/2152772}.
\newblock URL \url{http://dx.doi.org/10.2307/2152772}.

\bibitem[Milne(2005)]{milneintro}
J.~S. Milne.
\newblock {Introduction to {S}himura varieties}.
\newblock In \emph{{Harmonic analysis, the trace formula, and {S}himura
  varieties}}, volume~4 of \emph{{Clay Math. Proc.}}, pages 265--378. Amer.
  Math. Soc., Providence, RI, 2005.

\bibitem[Moonen(2004)]{BM}
Ben Moonen.
\newblock {Serre-{T}ate theory for moduli spaces of {PEL} type}.
\newblock \emph{Ann. Sci. {\'E}cole Norm. Sup. (4)}, 37\penalty0 (2):\penalty0
  223--269, 2004.
\newblock ISSN 0012-9593.
\newblock \doi{10.1016/j.ansens.2003.04.004}.
\newblock URL \url{http://dx.doi.org/10.1016/j.ansens.2003.04.004}.

\bibitem[Rapoport and Zink(1996)]{RZ}
M.~Rapoport and Th. Zink.
\newblock \emph{{Period spaces for {$p$}-divisible groups}}, volume 141 of
  \emph{{Annals of Mathematics Studies}}.
\newblock Princeton University Press, Princeton, NJ, 1996.
\newblock ISBN 0-691-02782-X; 0-691-02781-1.

\bibitem[Vollaard(2010)]{V1}
Inken Vollaard.
\newblock {The supersingular locus of the {S}himura variety for {${\rm
  GU}(1,s)$}}.
\newblock \emph{Canad. J. Math.}, 62\penalty0 (3):\penalty0 668--720, 2010.
\newblock ISSN 0008-414X.
\newblock \doi{10.4153/CJM-2010-031-2}.
\newblock URL \url{http://dx.doi.org/10.4153/CJM-2010-031-2}.

\bibitem[Vollaard and Wedhorn(2011)]{V2}
Inken Vollaard and Torsten Wedhorn.
\newblock {The supersingular locus of the {S}himura variety of {${\rm
  GU}(1,n-1)$} {II}}.
\newblock \emph{Invent. Math.}, 184\penalty0 (3):\penalty0 591--627, 2011.
\newblock ISSN 0020-9910.
\newblock \doi{10.1007/s00222-010-0299-y}.
\newblock URL \url{http://dx.doi.org/10.1007/s00222-010-0299-y}.

\bibitem[Wedhorn(2000)]{W1}
Torsten Wedhorn.
\newblock {Congruence relations on some {S}himura varieties}.
\newblock \emph{J. Reine Angew. Math.}, 524:\penalty0 43--71, 2000.
\newblock ISSN 0075-4102.
\newblock \doi{10.1515/crll.2000.060}.
\newblock URL \url{http://dx.doi.org/10.1515/crll.2000.060}.

\end{thebibliography}

\end{document}